\documentclass[12pt]{article}
\usepackage[english]{babel}
\usepackage{amssymb, graphicx}
\usepackage{amsmath}
\usepackage{amsthm,  mathrsfs, enumerate, cite}

\newcounter{mythm}[section]
\newtheorem{theorem}{Theorem}[section]
\newtheorem{lemma}[theorem]{Lemma}
\newtheorem{proposition}[theorem]{Proposition}
\newtheorem{corollary}[theorem]{Corollary}
\newtheorem{lf}{Lyapunov Function}
\newtheorem{claim}{Claim}
\theoremstyle{definition}
\newtheorem{definition}[mythm]{Definition}
\newtheorem{patch}{Patch}
\newtheorem{remark}{Remark}

\numberwithin{equation}{section}

\begin{document} 
\title{Transition from ergodic to explosive behavior in a family of stochastic differential equations}

\author{Jeremiah Birrell$^{a, 1}$, David P. Herzog$^{b, 2}$, Jan Wehr$^{b, 3}$ \\
\vspace{-.1in}
\scriptsize{$^{a}$Program in Applied Mathematics, The University of Arizona}\\ 
\scriptsize{617 N. Santa Rita Ave., P.O. Box 210089, Tucson, AZ 85721-0089, USA}\\
\vspace{-.1in}
\scriptsize{$^{b}$Department of Mathematics, The University of Arizona}\\
\scriptsize{617 N. Santa Rita Ave., P.O. Box 210089, Tucson, AZ 85721-0089, USA}\\
\vspace{-.1in}
\scriptsize{$^{1}$Corresponding author, (email) jbirrell@math.arizona.edu, (tel) 520-621-1963, (fax) 520-626-5048}\\ 
\scriptsize{$^{2}$(email) dherzog@math.arizona.edu, $^{3}$(email) wehr@math.arizona.edu} }

\maketitle

\begin{abstract}
We study a family of quadratic stochastic differential equations in the plane, motivated by applications to turbulent transport of heavy particles.  Using Lyapunov functions, we find a critical parameter value $\alpha_{1}=\alpha_{2}$ such that when $\alpha_{2}>\alpha_{1}$ the system is ergodic and when $\alpha_{2}<\alpha_{1}$ solutions are not defined for all times.  H\"{o}rmander's hypoellipticity theorem and geometric control theory are also utilized. 
\end{abstract}

\noindent \footnotesize{\textbf{Key Words:}  Ergodic Property, Stochastic Differential Equations, Degenerate Noise, Invariant (Probability) Measures, Geometric Control Theory, Lyapunov Functions.  }

\section{Introduction}

\normalsize
Understanding the ergodic behavior of both deterministic and random dynamical systems is of paramount importance in applications.  This is because many natural phenomena are either in steady state or close to it; thus the long-time dynamics often reflects most accurately the situation at hand.  Proving or disproving ergodicity in general, however, remains a challenge.  In this work
we discuss useful methods, as applied to a specific family of stochastic differential equations. 


As in \cite{HAS, MTIII, RB}, our main methodology is the use of Lyapunov functions.  These are widely utilized in the literature, e.g. \cite{EM, MSH, SM}, to study the behavior of diffusion processes.  It is well known that finding a Lyapunov function to verify global stability can be highly nontrivial.  We aim at a systematic approach to this problem by introducing the notion of a \emph{Lyapunov covering}.  This method consists of identifying the regions which the process must exit prior to leaving the state space.                          





Moreover, we hope that our use of geometric control theory \cite{AK, IW, JK, JK1, K2, SV}   
will elucidate the study of stochastic differential equations with degenerate noise.  One class of such systems arises by writing a second-order equation in first-order form \cite{GHW, MSH, RB}.  In such equations, noise is effectively transferred through the average dynamics so that the support of the diffusion for all positive times is the whole space.  In view of the support theorem \cite{SV}, this is proven by explicitly solving the associated control problem.  In the cases treated here, it is more convenient to use geometric methods to uncover the accessibility sets of the control system.  We find that, due to the even (quadratic) order of the coefficients of the average dynamics, noise only spreads to certain strict subsets of the state space.


In what follows, let $a_{1},\, a_{2}\in \mathbb{R}$, $\alpha_{1},\, \alpha_{2} > 0$ and $\kappa_{1}\geq 0$, $\kappa_{2}>0$.  We study the stochastic differential equation in $\mathbb{R}^{2}$:      
\begin{eqnarray}\label{initial_system}
dX_t&=&(a_1X_t-\alpha_1 X_t^2+Y_t^2)\, dt+\sqrt{2\kappa_1}\, dB^{(1)}_t \\
dY_t&=&(a_2Y_t-\alpha_2 X_tY_t) \, dt +\sqrt{2\kappa_2} \, dB^{(2)}_t, \nonumber
\end{eqnarray}
where $B^{(1)}_{t}$ and $B^{(2)}_{t}$ are independent standard Brownian motions.
Consider a two-dimensional spatially smooth turbulent flow.  It was noted in \cite{BCH} that in our system with $a_{1}=a_{2}=-1$ and $\alpha_{1}=1$, $\alpha_{2}=2$, $X_{t}$ and $Y_{t}$ model the transverse and longitudinal components of the velocity difference of two heavy particles transported by the flow.  Ergodicity of this equation was assumed in \cite{BCH} to extract information on particle clustering as expressed through the top Lyapunov exponent.  Subsequently, ergodicity was proven in \cite{GHW}.     Here, among other results, we provide an alternate proof of this fact. 

It was also noted in \cite{BCH, GHW} that, in this special case, equation \eqref{initial_system} can be written succinctly in the complex variable $Z_{t}=X_{t}+iY_{t}$ as:
\begin{eqnarray}\label{complex_system}
dZ_{t} &=& (-Z_{t} -Z_{t}^{2}) \, dt + \sqrt{2\kappa_{1}} \, dB^{(1)}_{t} + i \sqrt{2\kappa_{2}} \, dB^{(2)}_{t}.
\end{eqnarray}  
Here we show that ergodicity of equation \eqref{initial_system} holds whenever $\alpha_{2}>\alpha_{1}$; thus stability in \eqref{complex_system} above is not due to its holomorphic structure.  We shall see that this result is optimal in the sense that when $\alpha_{1}>\alpha_{2}$, there are solutions which reach infinity in finite time with positive probability.  It should be noted that the choice $\kappa_{1} \geq 0$, $\kappa_{2}>0$ is optimal as well.  For fixed $a_{1},\, a_{2}\in \mathbb{R}$ and $\alpha_{1}, \alpha_{2} > 0$, if $\kappa_{1} \geq 0,$ $\kappa_{2} =0$, with probability one equation \eqref{initial_system} has solutions which explode to infinity in finite time.  This is a simple consequence of Feller's test \cite{DR} applied to solutions starting along the negative real axis, hence we omit further discussion.

By working in this generality, our goal is to make progress in understanding turbulent transport of heavy particles in spatially rough flows as well.  In a two-dimensional flow with spatial H\"{o}lder exponent $h\in(0,1)$, the components, $X_t$ and $Y_t$, of the velocity difference now obey:   
\begin{eqnarray}\label{turbdyn}
dX_t &=&(-X_t - R_t^{-1}(h X_t^{2}-Y_t^{2}) )\, dt + \sqrt{2\kappa_1}\, dB^{(1)}_t\\
dY_t&=& (-Y_t -(1+h) R_t^{-1} X_t Y_t)\, dt + \sqrt{2\kappa_{2}}\, dB^{(2)}_t \nonumber\\
dR_t&=& (1-h) R_t \, dt \nonumber,
\end{eqnarray}  
where $R_t$ is proportional to $|S_t|^{1-h}$ where $S_t$ is the particle separation at time $t\geq 0$ \cite{BCH1}.  For simplicity, to gain insight into the dynamics \eqref{turbdyn} we assume $R_t$ is a positive constant; in which case, we see that equation \eqref{turbdyn} falls within the class \eqref{initial_system}.


The organization of the paper is as follows.  In Section \ref{results}, we fix notation, state the main results, and outline our methods of proof.  In Section \ref{lyapunov}, we show that when $\alpha_{2}>\alpha_{1}$, the process $(X_{t}, Y_{t})$ is nonexplosive and that equation \eqref{initial_system} has (at least one) invariant probability measure.  In Section \ref{positivity}, we prove uniqueness of the invariant probability measure using a combination of H\"{o}rmander's theorem and geometric control theory arguments.  In Section \ref{explosion}, we prove that when $\alpha_{1}>\alpha_{2}$, the process $(X_{t}, Y_{t})$ with initial conditions in a certain region goes to infinity in finite time with positive probability.


\section{Main Results}\label{results}


Let us first fix notation and terminology.  Throughout the paper $\mathcal{B}$ denotes the Borel $\sigma$-field of subsets on $\mathbb{R}^{2}$.  We assume the standard Brownian motions, $B^{(1)}_{t}$ and $B^{(2)}_{t}$, are defined on a common probability triple $(\Omega, \mathcal{F}, P)$.  $Z_{t}$ denotes the two-dimensional process $(X_{t}, Y_{t})$ defined by equation \eqref{initial_system}. $Z_{t}$ is defined on $(\Omega, \mathcal{F}, P)$ until the random time at which it leaves $\mathbb{R}^{2}$.  More formally, for $n\in \mathbb{N}$, if we define
\begin{eqnarray*}
\tau_{n} := \inf_{t>0} \{|Z_{t}| \geq n \},
\end{eqnarray*}
and let $\tau$ be the finite or infinite limit of $\tau_{n}$ as $n\rightarrow \infty$, $Z_{t}$ is defined on $(\Omega, \mathcal{F}, P)$ for all times $t \wedge \tau$, $t\geq 0$ (where here and in what follows $a \wedge b$ denotes the minimum of $a$ and $b$).  To emphasize that both $Z_{t}$ and $\tau$ depend on $Z_{0}=z\in \mathbb{R}^{2}$, the superscript notation $Z_{t}^{z}$ and $\tau^{z}$ is sometimes used \cite{HAS}.  We opt instead to account for the dependence on $z$ in the measure via $P_{z}$ and $E_{z}$.  

Using the notation above, $Z_{t\wedge \tau}$ is a Markov process with transition kernel
\begin{eqnarray*}
P(t, z, A):=P_{z}\left\{Z_{t\wedge \tau}\in A \right\}, 
\end{eqnarray*}
for $t\geq 0$, $z\in \mathbb{R}^{2}$, and $A\in \mathcal{B}$.  The measures $P(t, z, \, \cdot\, )$ give rise to a semigroup $\{P_{t}\}_{t\geq 0}$ which acts on real-valued bounded $\mathcal{B}$-measurable functions $\Phi$ via
\begin{eqnarray*}
P_{t} \Phi(z) = \int_{\mathbb{R}^{2}} P(t, z, dw)\Phi(w)
\end{eqnarray*}  
and on finite $\mathcal{B}$-measures $\mu$ via
\begin{eqnarray*}
\mu P_{t} (A) = \int_{\mathbb{R}^{2}} \mu(dw) P(t, w, A),
\end{eqnarray*}
for $A\in \mathcal{B}$.  The $\mathcal{B}$-measure $\mu$ is called an \emph{invariant measure} if for all $A\in \mathcal{B}$ and $t\geq 0$ 
\begin{eqnarray*}\label{invm}
\mu P_{t}(A)= \mu(A).  
\end{eqnarray*}
A positive invariant measure $\mu$ can be normalized to have total mass one; in which case, the resulting measure $\nu$ is a probability measure that satisfies $\nu P_{t}(A)=\nu(A)$ for all $A\in \mathcal{B}$, $t\geq 0$.  We call $\nu$ an \emph{invariant probability measure}.  It is well known that invariant probability measures represent the long-time behavior of the process $Z_{t\wedge \tau}$.  For a nice discussion of this, see \cite{RB}.       

Let 
\begin{eqnarray}\label{generator}
L = (a_{1}x-\alpha_{1}x^{2} + y^{2}) \frac{\partial}{\partial x} + (a_{2}y-\alpha_{2} xy) \frac{\partial}{\partial y}+ \kappa_{1} \frac{\partial^{2}}{\partial x^{2}} + \kappa_{2} \frac{\partial^{2}}{\partial y^{2}}.
\end{eqnarray}
For a suitable class of functions, e.g. $C^{2}$-functions $\Phi:\mathbb{R}^{2}\rightarrow \mathbb{R}$, for $n\in \mathbb{N}$ we have Dynkin's formula:
\begin{eqnarray*}
E_{z}\left[ \Phi(Z_{t\wedge \tau_{n}})\right]- \Phi(z)= E_{z} \left[ \int_{0}^{t\wedge \tau_{n}} L\Phi(Z_{s}) \, ds \right].   
\end{eqnarray*}
We call $L$ the \emph{generator} of the Markov process $Z_{t\wedge \tau}$.

We now state our main results.  The first says, provided $\alpha_{2}>\alpha_{1}$, $\tau = \infty$ almost surely regardless of the initial condition $z\in \mathbb{R}^{2}$.  Moreover, the Markov process $Z_{t}$, which is now defined for all finite times $t\geq 0$ and for all $Z_{0}=z\in \mathbb{R}^{2}$, has a unique invariant probability measure.  The second result shows that, in sharp contrast to the above, when $\alpha_{2}<\alpha_{1}$, there exists a non-empty set $A$ of initial conditions for which the process $Z_{t}$ with $Z_{0}=z\in A$ explodes with positive probability in finite time, i.e., $P_{z}\left\{\tau < \infty \right\}>0$ for $z\in A$.  The set $A$ will be explicitly described in Section \ref{explosion}.      


\begin{theorem}\label{theorem1}
If $\alpha_{2}>\alpha_{1}$, then:
\begin{description}
\item [(1)] For all $z\in \mathbb{R}^{2}$, 
\begin{eqnarray*}
P_{z}\left\{ \tau < \infty \right\}=0.  
\end{eqnarray*}  \noindent Thus for all $z\in \mathbb{R}^{2}$, $P_{z}\left\{Z_{t}=Z_{t\wedge \tau} \text{ for all } t\geq 0\right\}=1$.   
\item [(2)] There exists a unique invariant probability measure $\nu$. 
\end{description}
\end{theorem}    

\begin{theorem}\label{theorem2}
If $\alpha_{1}>\alpha_{2}$, then there exists $\emptyset \neq A \subset \mathbb{R}^{2}$ such that for all $z\in A$, 
\begin{eqnarray*}
P_{z}\left\{\tau <\infty \right\}>0.
\end{eqnarray*}

\end{theorem}

Theorem \ref{theorem1} will be established by proving Lemma \ref{lem1} and Lemma \ref{lem2} below.  Lemma \ref{lem1} implies that $Z_{t}$ is nonexplosive and has at least one invariant probability measure (see \cite{HAS, MTIII, RB}).  Lemma \ref{lem2}, which is only necessary when $\kappa_{1}=0$, shows that the supports of any two \emph{extremal}\footnote{An invariant probability measure $\nu$ is \emph{extremal} if whenever $\nu = (1-\lambda) \nu_{1} + \lambda \nu_{2}$ for $\nu_{1}, \nu_{2}$ invariant probability measures and $\lambda \in (0,1)$, then $\nu_{1}=\nu_{2}=\nu.$} invariant probability measures contain a non-empty open subset in common.  This, coupled with regularity of the transition measures of the process $Z_t$, proves uniqueness.  For more details and further information, see \cite{AK, HAS, HM, MTI, MTII, MTIII}.  

\begin{lemma}\label{lem1}
If $\alpha_{2}>\alpha_{1}$, there exists a $C^{\infty}$-function $\Phi :\mathbb{R}^{2}\rightarrow \mathbb{R}$ such that 
\begin{description}
\item[(1)] $\Phi(z) \rightarrow \infty$ as $|z| \rightarrow \infty$.
\item[(2)]  There exist positive constants $C, D$ such that 
\begin{eqnarray}\label{boundL}
L\Phi(z) & \leq &-C\Phi(z)+ D,
\end{eqnarray}
for all $z \in \mathbb{R}^{2}$.  
\end{description}
\end{lemma}

\begin{lemma}\label{lem2}
~~~

\begin{description}
\item[(1)] For $t>0$, $z\in \mathbb{R}^{2}$:
\begin{equation*}
P(t, z, dw)= p(t, z, w) \,dw,
\end{equation*}
where $p(t, z, w)$ is jointly $C^{\infty}$ in $(t, z,w)\in (0, \infty) \times \mathbb{R}^{2}\times \mathbb{R}^{2}$ and $dw$ is Lebesgue measure on $\mathbb{R}^{2}$.  Moreover, any invariant probability measure $\nu$ (such measures exist by Lemma \ref{lem1}) satisfies
\begin{equation*}
\nu(dw) = \rho_{\nu}(w) \, dw,
\end{equation*}   
where $\rho_{\nu}\in C^{\infty}(\mathbb{R}^{2})$.  
\item[(2)]  Let $\mu, \tilde{\mu}$ be extremal invariant probability measures.  Then there exists a non-empty open $U$ such that 
\begin{eqnarray*}
\text{supp }\mu\cap \text{supp } \tilde{\mu} \supset U.
\end{eqnarray*}

\end{description} 
\end{lemma}

To establish Theorem \ref{theorem2}, we use the following lemma, the proof of which is a simple consequence of Dynkin's formula.  For more information, consult \cite{HAS}.    

\begin{lemma}\label{psiexp}
There exist a bounded $C^{\infty}$-function $\Psi: \mathbb{R}^{2}\rightarrow [0, \infty)$, strictly positive on $A\neq \emptyset \subset \mathbb{R}^{2}$, and a constant $E>0$ such that:
\begin{eqnarray*}
L \Psi(z) &\geq & E \Psi(z),
\end{eqnarray*}
for all $z\in \mathbb{R}^{2}.$
\end{lemma}

We note that in the case when $\kappa_{1}>0$ and $\alpha_{2}>\alpha_{1}$, we can prove that the rate of convergence to the invariant probability measure $\nu$ is exponential.  To see this, first note that we may choose a constant $F>0$ such that $\Phi_{F}=\Phi + F \geq 1$.  The rate can then be quantified by the norm $\| \, \cdot \,\|_{\Phi_{E}}$ defined  on $\mathcal{B}$-measures $\mu$ via:  
\begin{equation*}
\| \mu \|_{\Phi_{F}}= \sup_{0\leq f \leq \Psi_{F}}  \left|\int f(y) \, d\mu(y) \right|, 
\end{equation*} 
where the supremum is taken over nonnegative $\mathcal{B}$-measurable real-valued functions $f$ bounded above by $\Phi_{F}$ \cite{HM, MTIII}.

\begin{corollary}
If $\kappa_{1}>0$, then there exist $r, C>0$ such that
\begin{equation}\label{expconv}
\| P(t, z, \, \cdot \,) - \nu \|_{\Phi_{F}} \leq C \Phi_{F}(z)e^{-rt},
\end{equation}
for all $t\geq 0$, $z\in \mathbb{R}^{2}$.  Moreover, since $\Phi_{F}\geq 1$ we can replace the left-hand side of \eqref{expconv} by $\| P(t, z, \, \cdot \, )-\nu\|_{TV}$ where $\| \,\cdot \,\|_{TV}$ is the total variation norm on $\mathcal{B}$-measures.  \end{corollary}

\begin{proof}
The second conclusion easily follows from the first.  To see the first, since both $\kappa_{1}$ and $\kappa_{2}$ are strictly positive, it follows that 
\begin{eqnarray*}
\text{supp }P(t, z, \,\cdot\,) = \mathbb{R}^{2},
\end{eqnarray*}
for all $t>0$, and all $z\in \mathbb{R}^{2}$.  Moreover, Lemma \ref{lem2} \textbf{(1)} is satisfied.  This  together with Lemma \ref{lem1} proves the result by Theorem A.2 of \cite{EM}.  
\end{proof}




\section{Existence of a Lyapunov Function}\label{lyapunov}

In this section, we prove Lemma \ref{lem1}.  Unfortunately, it is difficult to write down a single formula which would define $\Phi$ for all $z$.  There are a number of reasons for this.  First, when $\kappa_{1}=\kappa_{2}=0$, there is an unstable trajectory in \eqref{initial_system} along the negative $X$-axis.  Thus some noise, e.g. $\kappa_{2}>0$, must be utilized so that solutions do not explode.  In terms of estimating $L\Phi$, this means second order terms must be used along this trajectory.  Second, it is clear that noise with arbitrarily small magnitude in the $Y$-direction will direct $Z_{t}$ with $Z_{0}=(X_{0}, Y_{0})=(x,0),\, x<0$ off the unstable trajectory.  Once the process leaves the axis, the same noise that initially helped could now steer the process back to the unstable trajectory.  This is reflected in the estimates for $L\Phi$ in a sufficiently large region which includes the negative $X$-axis.  And third, once the process exits this region we must assure stability in the remaining regions as well.  Because the nature of the unperturbed dynamics changes significantly from the previous region, this calls naturally for a function $\Phi$ of a different type.  We now introduce the following definitions to elucidate our procedure.

\begin{definition}\label{lyapunovfunction}
Let \(U\subset \mathbb{R}^2\) be an unbounded region such that $\partial U$ is a continuous, piecewise smooth curve.  We call a function $\phi: U \rightarrow \mathbb{R}$ a \emph{Lyapunov function (for the operator $L$) on }\(U\) if: 
\begin{enumerate}[(I)]
\item $\phi\in C^{\infty}(U)$.   
\item \(\phi(z)\rightarrow\infty\) as \(|z|\rightarrow\infty,\ z\in U\).
\item There exist $C, D>0$ such that \(L\phi(z)\leq -C\phi(z)+ D\) for all $z\in U$.  
\end{enumerate}  
\end{definition}

\begin{definition}
Let $\phi_{1}, \phi_{2}, \ldots, \phi_{N}$ be Lyapunov functions on $U_{1}$, $U_{2}$, $\ldots$, $U_{N}$ respectively and suppose there exists $R>0$ such that
\begin{eqnarray*}
\mathbb{R}^{2}=\bigcup_{j=1}^{N} U_{j}\cup B_{R},
\end{eqnarray*}
where $B_{R}$ is the open ball of radius $R$ centered at the origin in $\mathbb{R}^{2}$.  We call $\{ (\phi_{1}, U_{1}), (\phi_{2}, U_{2}), \ldots, (\phi_{N}, U_{N})\}$ a \emph{Lyapunov covering}.    
\end{definition}

To exhibit $\Phi$, we will first find a Lyapunov covering.  The existence of such a covering is of course not sufficient to prove Lemma \ref{lem1}.  As suggested by the following proposition, however, finding a Lyapunov covering is a step in the right direction.  

\begin{proposition}\label{exitprop}
Suppose $\phi$ is a Lyapunov function on $U$.  Let $V \subset U$ be another unbounded region, satisfying the same assumptions as $U$ and such that $\partial U \cap \partial V=\emptyset$.  If $\tau_{V}=\inf_{t>0}\{ Z_{t} \in V^{c}\}$, then for $z\in V$
\begin{eqnarray*}
P_{z} \left\{ \tau_{V} < \tau \right\}=1.  
\end{eqnarray*}          
\end{proposition} 

\begin{proof}  Let $\tau_{U}=\inf_{t>0} \{Z_{t}\in U^{c} \}$ and $\tau_{n, U}(t) =\tau_{n} \wedge \tau_{U} \wedge t$.  By adding a sufficiently large constant to $\phi$, we may assure the resulting function is a non-negative Lyapunov function on $U$.  Let us use $\phi$ to denote this non-negative function.  It is clear that the much weaker bound $L\phi(z) \leq C\phi(z) + D$ is satisfied on $U$.  Letting $\psi(z, t)=e^{-Ct} (\phi+\frac{D}{C})$, for $z\in V$ we obtain by Dynkin's formula 
\begin{eqnarray*}
E_{z}\left[ \psi(Z_{\tau_{n, U}(t)}, \tau_{n, U}(t))\right]-\psi(z,0)&=& E_{z} \left[\int_{0}^{\tau_{n, U}(t)}\psi_{t}(s, Z_{s})+L \psi(s, Z_{s}) \, ds \right]\\
&\leq& 0.  
\end{eqnarray*}
We see that
\begin{eqnarray*}
E_{z}\left[ \psi(Z_{\tau_{n, U}(t)}, \tau_{n, U}(t))\right]&\geq &e^{-Ct} \phi_{n} P_{z}\left\{\tau_{n} \leq \tau_{U}\wedge t  \right\},
\end{eqnarray*} 
where $\phi_{n}=\inf_{z\in U \cap \partial B_{n}} \phi(z) \rightarrow \infty$ as $n\rightarrow \infty$.  
Combining the previous two estimates yields
\begin{eqnarray*}
 P_{z}\left\{\tau_{n} \leq \tau_{U}\wedge t  \right\} \leq \phi_{n}^{-1} e^{Ct} \psi(z, 0).  
\end{eqnarray*}
Letting $n\rightarrow \infty$ we see that $P_{z} \{ \tau \leq \tau_{U} \wedge t\}=0$ for all $t\geq 0$, whence $P_{z} \{ \tau \geq \tau_{U}\}=1$.  By path continuity of $Z_{t}$, $P_{z}\{\tau_{U} > \tau_{V}\}=1$.  Hence $P_{z}\{ \tau > \tau_{V}\}=1$ as claimed.      
\end{proof}

\begin{figure}
\caption{Cartoon of Proposition \ref{exitprop}}
\centering
\includegraphics[width=4 in]{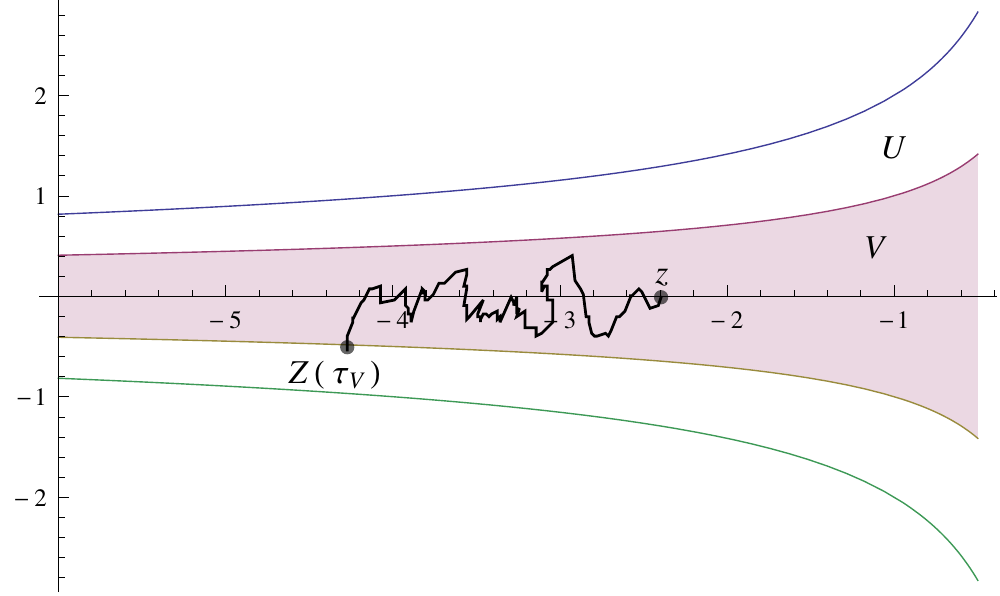}
\end{figure}
\begin{remark}
Note that the same behavior is true if Property (III) of Definition \ref{lyapunovfunction} is replaced by a much weaker bound.  Indeed suppose $U$ is as in Definition \ref{lyapunovfunction} and $\phi:U \rightarrow \mathbb{R}$ satisfies (I) and (II) of Definition \ref{lyapunovfunction}.  If $\phi$ satisfies the bound
\begin{equation*}
L \phi(z) \leq C \phi(z) + D,
\end{equation*}
for all $z\in \mathbb{R}^{2}$ for some $C, D>0$, then the conclusion of Proposition \ref{exitprop} remains valid.  
\end{remark}

The existence of a Lyapunov covering $\{ (\phi_{j}, U_{j}\}_{j=1}^{N}$ guarantees that $Z_{t}$ cannot go to infinity directly through any unbounded region $V\subset U_{j}$ with continuous boundary such that $\partial V \cap \partial U_{j}=\emptyset$.  As we shall see in our case, it is possible to extract subsets $V_{j}\subset U_{j}$ with continuous piecewise smooth boundaries such that $\partial U_{j} \cap \partial V_{j}=\emptyset$ for all $j=1,2,\ldots, N$ and $\{(\phi_{j}, V_{j})\}_{j=1}^{N}$ is a Lyapunov covering.  

\begin{definition}
We call any $\{ (\phi_{j}, V_{j})\}_{j=1}^{N}$ with the above properties \emph{a strong Lyapunov covering subordinate to }$\{(\phi_{j}, U_{j})\}_{j=1}^{N}$.   
\end{definition}   

Hence, given the existence of a strong Lyapunov covering $\{(\phi_{j}, V_{j}) \}_{j=1}^{N}$, $Z_{t}$ cannot go to infinity staying inside one of the $V_{j}$.  Moreover 
\begin{eqnarray*}
\mathbb{R}^{2}=\bigcup_{j=1}^{N} V_{j} \cup B_{R},
\end{eqnarray*} 
for some $R>0$.  Thus it seems likely that $Z_{t}$ cannot diverge in finite time.  However, it is possible that $Z_{t}$ oscillates between two or more regions on its way to infinity in finite time.  To eliminate this possibility we glue our Lyapunov covering together so that we have a globally-defined Lyapunov function $\Phi$. 

The construction of $\Phi$ is done in three stages:
 \begin{description}
 \item[Stage 0:] Reduction of parameters.
 \item[Stage 1:] Existence of a strong Lyapunov covering.
 \item[Stage 2:] Existence of a globally-defined Lyapunov function $\Phi$.  
 \end{description}   


\subsection{Stage 0}

Equation \eqref{initial_system} depends on $a_{1}, a_{2}\in \mathbb{R}$, $\alpha_{1}, \alpha_{2} > 0$, and $\kappa_{1} \geq 0$, $\kappa_{2}>0$.  We thus may expect our globally-defined Lyapunov function to depend on these parameters as well.  To simplify matters, let $\boldsymbol{a}=(a_{1}, a_{2})$, $\boldsymbol{\alpha}=(\alpha_{1}, \alpha_{2})$, $\boldsymbol{\kappa}=(\kappa_{1}, \kappa_{2})$, and $L^{\boldsymbol{\kappa}}_{\boldsymbol{a}, \boldsymbol{\alpha}}$ be the generator \eqref{generator} of $Z_{t}$ that solves equation \eqref{initial_system} with those parameter values.  Define
\begin{eqnarray*}
\mathcal{N}(\boldsymbol{a}, \boldsymbol{\alpha}, \boldsymbol{\kappa})=\left\{\Phi\,:\, \Phi \text{ is a Lyapunov function on } \mathbb{R}^{2} \text{ for } L^{\boldsymbol{\kappa}}_{\boldsymbol{a}, \boldsymbol{\alpha}}\right\}.
\end{eqnarray*}   

We have the following:

\begin{proposition}\label{vary_kappa_lemma}
Let $\alpha_{2}>\alpha_{1}>0$ and suppose there exists \(\kappa_2>0\) such that for all $\boldsymbol{a}\in \mathbb{R}^{2}$ and $\kappa_{1}\geq 0$, \(\mathcal{N}(\boldsymbol{a},\boldsymbol{\alpha},\boldsymbol{\kappa})\neq \emptyset\).  Then \(\mathcal{N}(\boldsymbol{b},\boldsymbol{\alpha},\boldsymbol{\iota})\neq \emptyset \) for all $\boldsymbol{b}\in \mathbb{R}^{2}$, $\boldsymbol{\iota}\in \mathbb{R}_{\geq 0} \times \mathbb{R}_{>0}$.
\end{proposition}
\begin{proof}
Let $\boldsymbol{\iota}=(\iota_{1}, \iota_{2})\in  \mathbb{R}_{\geq 0} \times \mathbb{R}_{>0}$ and $\boldsymbol{b}\in \mathbb{R}^{2}=(b_{1}, b_{2})$ be arbitrary.  Define \(c=(\kappa_2\iota_2^{-1})^{\frac{1}{3}}\), \(\boldsymbol{a}=c\boldsymbol{b}\), and \(\boldsymbol{\kappa}=(c^3\iota_1,\kappa_2)\).  By assumption, there exists \(\Psi\in \mathcal{N}(\boldsymbol{a},\boldsymbol{\alpha},\boldsymbol{\kappa})\).  Define \(\Phi(z)=\Psi(cz)\).  It is easy to see that $\Phi$ satisfies Properties (I) and (II) of Definition \ref{lyapunovfunction} in all of $\mathbb{R}^{2}$.  To see (III), we use the  chain rule to obtain:
\begin{eqnarray*}
 L_{\boldsymbol{b},\boldsymbol{\alpha}}^{\boldsymbol{\iota}}\Phi(z)&=&\frac{1}{c}(L_{\boldsymbol{a},\boldsymbol{\alpha}}^{\boldsymbol{\kappa}}\Psi)(cz)\\
 &\leq& -C\Psi(cz)+ D\\
 &=& -C \Phi(z)+D,
\end{eqnarray*}
for some $C, D>0$.    
\end{proof} 

Using Proposition \ref{vary_kappa_lemma}, it is enough to show, given $\alpha_{2}>\alpha_{1}$, that there exists $\kappa_{2}>0$ sufficiently small such that $\mathcal{N}(\boldsymbol{a}, \boldsymbol{\alpha}, \boldsymbol{\kappa}) \neq \emptyset$ for all $\boldsymbol{a}\in \mathbb{R}^{2}$ and $\kappa_{1}\geq 0$.  This is equivalent to saying that for all $\alpha_{2}> \alpha_{1}$, there exists $\kappa_{2}=\kappa_{2}(\boldsymbol{\alpha})>0$ and $\Phi^{\boldsymbol{\kappa}}_{\boldsymbol{a}, \boldsymbol{\alpha}}$ such that $\Phi^{\boldsymbol{\kappa}}_{\boldsymbol{a}, \boldsymbol{\alpha}}\in \mathcal{N}(\boldsymbol{a}, \boldsymbol{\alpha}, \boldsymbol{\kappa})$ for all $\boldsymbol{a}\in \mathbb{R}^{2}$ and all $\kappa_{1}\geq 0$.  To obtain such a function $\Phi^{\boldsymbol{\kappa}}_{\boldsymbol{a}, \boldsymbol{\alpha}}$, we first find a Lyapunov covering.          


\subsection{Stage 1}

We will now construct families of strong Lyapunov coverings that depend on $\boldsymbol{\alpha}, \, \boldsymbol{a},$ and $\boldsymbol{\kappa}$.  By Proposition \ref{vary_kappa_lemma}, we will fix $\alpha_{2}>\alpha_{1}$ and pick $\kappa_{2}=\kappa_{2}(\boldsymbol{\alpha})>0$ such that $(\phi_{i}, U_{i})$ for $i=1,2,\ldots, 5$ below (which will depend on $\boldsymbol{\alpha}$, $\boldsymbol{a}$, and $\boldsymbol{\kappa}$) is a Lyapunov covering for all $\boldsymbol{a}\in \mathbb{R}^{2}$ and all $\boldsymbol{\kappa} \in \mathbb{R}_{\geq 0} \times (0, \kappa_{2}(\boldsymbol{\alpha})]$.  

To allow further flexibility in Stage 1 and Stage 2 of the procedure, there is a multitude of parameters.  We have thus provided a list of all parameters and their chosen values in the appendix.  Moreover, to illustrate the regions $U_{i}$ for $i=1,2,\ldots 5$, we have provided Figures 3.2-3.5.  In what follows, the most crucial ingredient is $(\phi_{1}, U_{1})$ where $U_{1}$ covers the unstable trajectory along the negative $X$-axis.


\begin{lf}\label{construct_phi1}
Define
\begin{eqnarray*}
\phi_1(x,y)=C_1(5|x|^\beta-y^2|x|^{\beta+1})
\end{eqnarray*}
on 
\begin{eqnarray*}
U_1=\{x<-2\}\cap\left\{|y|<2|x|^{-1/2}\right\}.
\end{eqnarray*}
Then $\phi_{1}$ is a Lyapunov function on $U_{1}$ for all $\boldsymbol{a}\in \mathbb{R}^{2}$ and all $\kappa_{1}\geq 0$.  

\end{lf}
\begin{proof}
The region excludes the $y$-axis, so \(\phi_1\in C^{\infty}(U_{1})\).  Note by the choice of $U_{1}$, \(\phi_1(x,y)>C_1 |x|^\beta\), therefore \(\phi_1(x,y)\rightarrow \infty\) as \((x,y)\rightarrow \infty,\, (x,y)\in U_{1}$.  Thus (I) and (II) are satisfied in $U_{1}$.  To see (III), note that:
\begin{eqnarray*}
L\phi_1 (x,y) \leq-C_1(2\kappa_2-5\alpha_1 \beta)|x|^{\beta+1}+ \mathcal{O}(|x|^{\beta}).
\end{eqnarray*}
Therefore there exists $D_{1}>0$ such that:
\begin{eqnarray} \label{estimate1}
L\phi_1(x,y)&\leq&-\frac{C_1}{2}(2\kappa_2-5 \alpha_1\beta)|x|^{\beta+1} +D_{1} \\
&\leq&-\frac{1}{10}(2\kappa_2-5\alpha_1 \beta)\phi_1(x,y) + D_{1}\nonumber\\
&\leq &-\frac{\kappa_{2}}{80}\phi_{1}(x,y) + D_{1},\nonumber 
\end{eqnarray}
where the last inequality follows by the choice (see the Appendix) of $\sigma \in (\alpha_{1}/\alpha_{2}, 1)$, $\delta=\frac{\kappa_{2}}{8\alpha_{1}}$, and $\beta = (2+\sigma) \delta$.  
\end{proof}

\begin{figure}\label{fig2}
\caption{The regions $U_{1}$ and $U_{2}$, along with their intersection $U_{1}\cap U_{2}$.}
\centering
\includegraphics[width=4.5in]{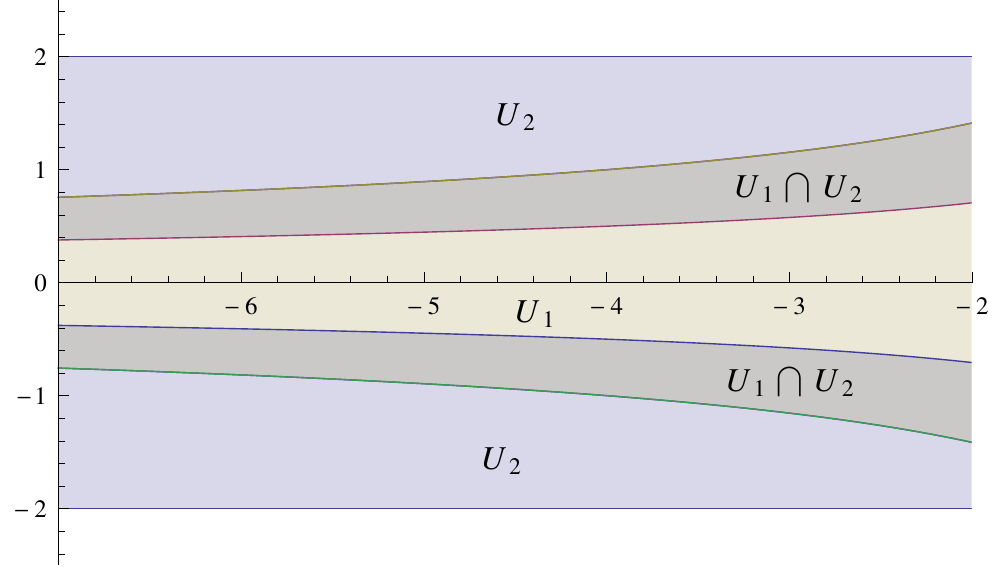}
\end{figure}

\begin{lf}\label{construct_phi2}
Define
\begin{eqnarray*}
\phi_2(x,y)=C_2\frac{|x|^{2\delta}+|y|^{2\delta}}{|y|^{2\sigma\delta}}
\end{eqnarray*}
on 
\begin{eqnarray*}
U_2=\{x<-2\}\cap\left\{ |x|^{-1/2} <|y|<2\right\}.
\end{eqnarray*}
Then $\phi_{2}$ is a Lyapunov function on $U_{2}$ for all $\boldsymbol{a}\in \mathbb{R}^{2}$ and all $\kappa_{1} \geq 0$.  
\end{lf}
\begin{proof}
This region excludes both $x$ and $y$-axes, therefore $\phi_2\in C^{\infty}(U_{2})$.  Note that \(\phi_2 (x,y) >4^{-\sigma\delta}C_2|x|^{2\delta}\) on \(U_2\), whence \(\phi_2 (x,y)\rightarrow \infty\) as \((x,y)\rightarrow \infty,\, (x,y) \in U_{2}\).  Thus $\phi_{2}$ satisfies (I) and (II) on $U_{2}$.  To see (III), note that:  
\begin{eqnarray}\label{estimate2p}
\nonumber L\phi_2(x,y)&=&2C_2\delta\frac{|x|^{2\delta + 1}}{|y|^{2\sigma \delta}}\left((\alpha_{1}-\sigma \alpha_{2}) + \kappa_{2} \sigma (2 \sigma \delta + 1) |x|^{-1} |y|^{-2} + o(1) \right)\\
&\leq& 2C_2\delta\frac{|x|^{2\delta + 1}}{|y|^{2\sigma \delta}}\left((\alpha_{1}-\sigma \alpha_{2}) + \kappa_{2} \sigma (2 \sigma \delta + 1)  + o(1) \right)
\end{eqnarray}
By the choice of $\kappa_{2}$,
\begin{eqnarray*}
(\alpha_{1}-\sigma \alpha_{2}) + \kappa_{2} \sigma( 2\sigma \delta + 1) <0.  
\end{eqnarray*}
Therefore there exists $D_{2}>0$ such that 
\begin{eqnarray}\label{estimate2}
L\phi_2(x,y)&\leq &  \delta ( (\alpha_{1}-\sigma \alpha_{2}) + \kappa_{2} \sigma ( 2\sigma \delta + 1)) \phi_{2}(x,y) + D_{2}.  
\end{eqnarray}
\end{proof}

\begin{lf}\label{construct_phi3}
Define 
\begin{eqnarray*}
\phi_3(x,y)=C_3\left(\frac{Dx^2+y^2}{|y|^{2\sigma}}\right)^\delta
\end{eqnarray*}
on 
\begin{eqnarray*}
U_3=\{x<-1/2\}\cap\{|y|>1\}.
\end{eqnarray*}
For all $\boldsymbol{a}\in \mathbb{R}^{2}$ and all $\boldsymbol{\kappa}\in \mathbb{R}^{2}$,  $\phi_{3}$ is a Lyapunov function in $U_{3}$.  
\end{lf}
\begin{proof}
This region excludes the \(x\) and \(y\)-axes, so \(\phi_3\in C^{\infty}(U_{3})\).  Moreover, it is clear that \(\phi_3 (x,y) \rightarrow \infty\) as \((x,y)\rightarrow\infty, \, (x,y) \in U_{3}\).  Thus $\phi_{3}$ satisfies (I) and (II) in $U_{3}$.  Since $\delta\in (0,1)$, we note that the $\delta(\delta-1)$ terms in $\partial_{xx}\phi_{3}(x,y)$ and $\partial_{yy} \phi_{3}(x,y)$ are negative.  Thus we have
\begin{eqnarray}\label{estimate2p1}
\nonumber L \phi_{3}(x,y)&\leq & 2 \delta \phi_{3}(x,y) \bigg{[} (a_{1} x-\alpha_{1} x^{2} + y^{2}) \frac{Dx}{Dx^{2}+y^{2}} \\
\nonumber &\, & \, + \, (a_{2}y-\alpha_{2}xy)\bigg{(}\frac{-\sigma D \text{sgn}(y) x^{2} + (1-\sigma) \text{sgn}(y) y^{2}}{|y|(Dx^{2}+y^{2})} \bigg{)}\\
\nonumber &\, & \, + \, \frac{\kappa_{2} \sigma (2\sigma+1) Dx^{2}}{y^{2}( Dx^{2}+y^{2})} + o(1)\bigg{]}\\
&\leq &2 \delta \phi_{3}(x,y) \bigg{[} D(\alpha_{1}-\sigma \alpha_{2}) \frac{|x|^{3}}{Dx^{2}+y^{2}}\\
\nonumber &\,& \, + \,\big{(}(D-\alpha_{2}(1-\sigma)) x + |a_{2}| (1-\sigma)\big{ )}\frac{y^{2}}{Dx^{2}+y^{2}}\\
\nonumber &\,& \, + \, \mathcal{O}\left(\frac{x^{2}}{Dx^{2}+y^{2}} \right)+ o(1)\bigg{]}.  
\end{eqnarray}
By the choice of $\sigma \in (\alpha_{1}/\alpha_{2}, 1)$ and $D> \max\left(1, (\alpha_{2}+ 2|a_{2}|)(1-\sigma)\right)$ there exist constants $C_{3}', D_{3}>0$ such that 
\begin{eqnarray*}
L\phi_{3}(x,y) &\leq & -C_{3}' \phi_{3}(x,y) + D_{3},
\end{eqnarray*}
for $(x,y) \in U_{3}$.  This proves (III) is valid for $\phi_{3}$ in $U_{3}$.  
\end{proof}

\begin{figure}
\caption{The regions $U_{2}$ and $U_{3}$}
\centering
\includegraphics[width=4in]{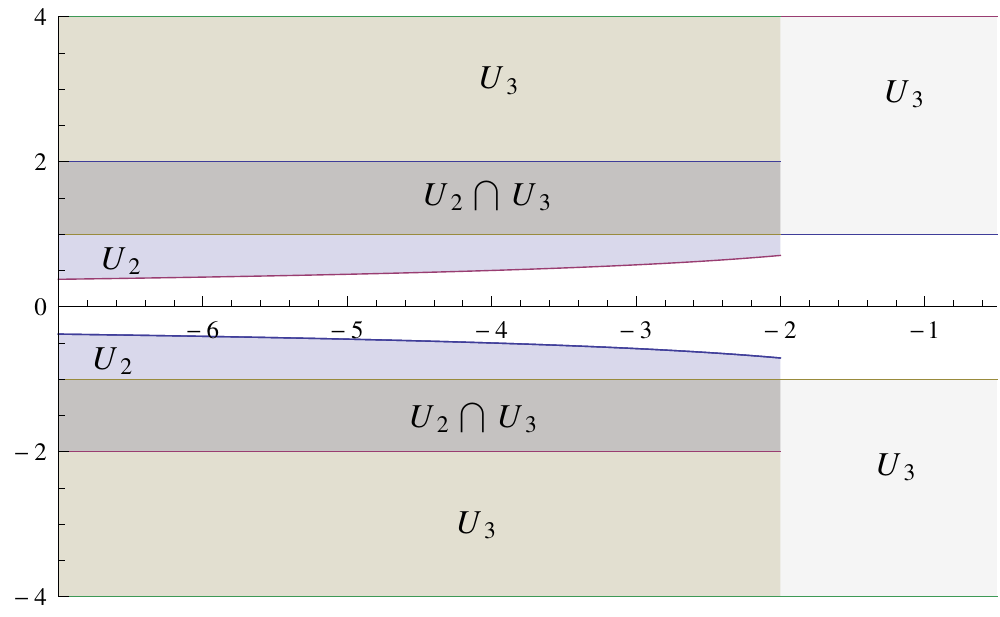}
\end{figure}

\begin{lf}\label{construct_phi4}
Define
 \begin{eqnarray*}
 \phi_4(x,y)=C_4\left(-x+N|y|^{2(1-\sigma)\delta}\right)
 \end{eqnarray*}
on 
\begin{eqnarray*}
U_4=\{-1<x<N\}\cap\{|y|>1\}.
\end{eqnarray*}
For all $\boldsymbol{a}, \, \boldsymbol{\kappa}\in \mathbb{R}^{2}$, $\phi_4$ is a Lyapunov function in $U_{4}$.  
\end{lf}
\begin{proof}
This region excludes the \(x\)-axis, so \(\phi_4\in C^{\infty}(U_{4})\).  As \((x,y)\rightarrow\infty\) with $(x,y) \in U_{4}$, \(|y|\rightarrow\infty\) since \(x\) is bounded. Therefore \(\phi_4(x,y)\rightarrow\infty\) as $(x,y) \rightarrow \infty,\, (x,y) \in U_{4}$.  Thus $\phi_{4}$ satisfies (I) and (II) in $U_{4}$.  To see (III) is valid in $U_{4}$, note that
\begin{eqnarray*}
L\phi_4(x,y)&\leq& -C_4y^2+O(|y|^{2(1-\sigma)\delta}).  
\end{eqnarray*}
Note this implies that there exist $C_{4}', D_{4}>0$ such that 
\begin{eqnarray*}
L\phi_{4}(x,y) &\leq & -C_{4}' \phi_{4}(x,y) + D_{4},
\end{eqnarray*}
for $(x,y) \in U_{4}$.  

\end{proof}

\begin{figure}
\caption{The regions $U_{3}$ and $U_{4}$ with $N\geq 1$.}
\centering
\includegraphics[width=4in]{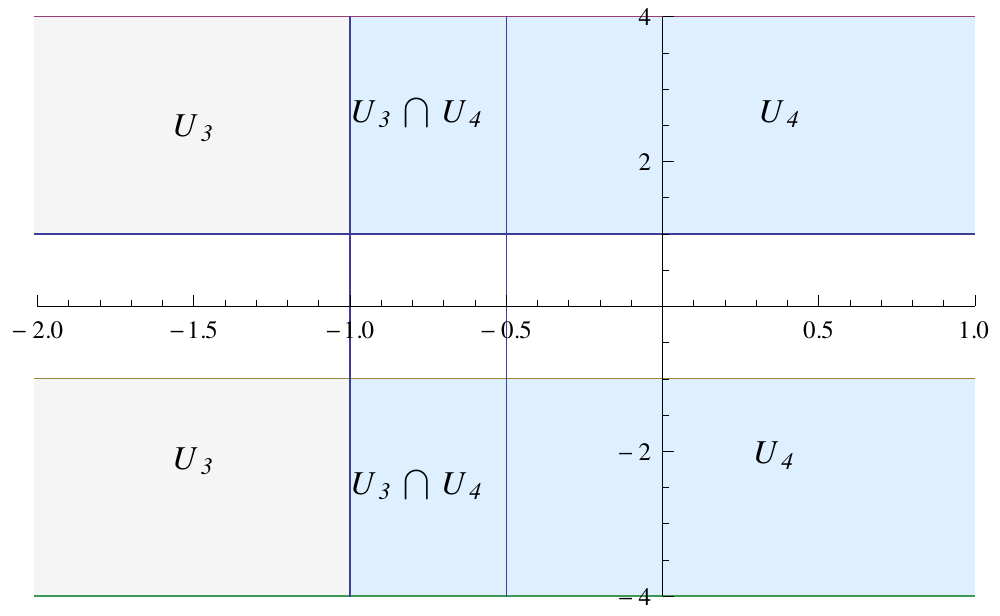}
\end{figure}

\begin{lf}\label{construct_phi5}
Define
\begin{eqnarray*}
\phi_5(x,y)=C_5(\eta x^2+y^2)^\gamma
\end{eqnarray*}
on
\begin{eqnarray*}
U_5=\{x>N/2\}.  
\end{eqnarray*}
For all $\boldsymbol{a}, \, \boldsymbol{\kappa}\in \mathbb{R}^{2}$, $\phi_{5}$ is a Lyapunov function on $U_{5}$.
\end{lf}
\begin{proof}
This region excludes the origin, so \(\phi_5\in C^{\infty}(U_{5})\).  Moreover \(\phi_5\rightarrow \infty\) as \((x,y)\rightarrow \infty\).  Thus $\phi_{5}$ satisfies (I) and (II) in $U_{5}$.  Since $\gamma\in (0,1)$ we see that the $\gamma(\gamma-1)$ terms in $\partial_{xx} \phi_{5}(x,y)$ and $\partial_{yy} \phi_{5}(x,y)$ are negative.  Thus we have
\begin{eqnarray*}
L \phi_{5}(x,y) &\leq &2 \gamma\phi_{5}(x,y) \bigg{[} (a_{1}x - \alpha_{1}x^{2}+y^{2}) \frac{\eta x}{\eta x^{2}+y^{2}}\\
&\,& \, +\, (a_{2} y -\alpha_{2} x y ) \frac{y}{\eta x^{2}+y^{2}} + o(1)\bigg{]}\\
&\leq & 2\gamma \phi_{5}(x,y) \bigg{[} -\frac{\alpha_{1}\eta x^{3}}{\eta x^{2}+y^{2}}- \big{(}(\alpha_2-\eta)x-|a_2|\big{)} \frac{y^{2}}{\eta x^{2}+y^{2}}\\
&\,& \, + \, \frac{|a_{1}|\eta x^{2}}{\eta x^{2}+y^{2}}+ o(1)\bigg{]}.  
\end{eqnarray*}
By the choice of $\eta=\alpha_2/2$ and $N>4|a_{2}|/\alpha_{2}$, there exist constants $C_{5}', D_{5}>0$ such that
\begin{eqnarray*}
L\phi_{5}(x,y) &\leq & -C_{5}' \phi_{5}(x,y) + D_{5},
\end{eqnarray*}
for $(x,y) \in U_{5}$.  Hence $\phi_{5}(x,y)$ satisfies (III) in $U_{5}$.    
\end{proof}

\begin{figure}
\caption{The regions $U_{4}$ and $U_{5}$ with $N=2$.}
\centering
\includegraphics[width=4in]{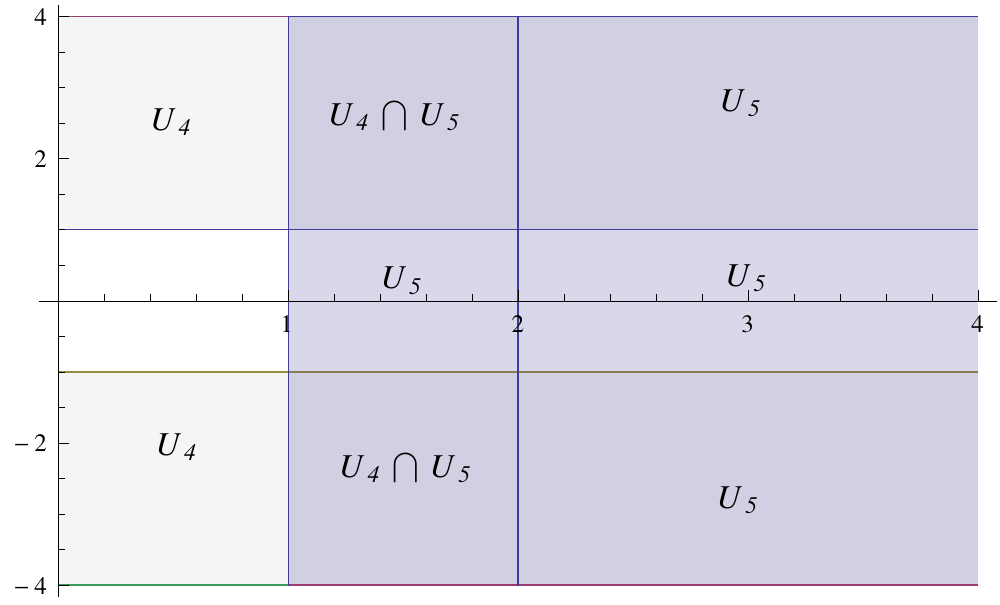}
\end{figure}

By construction, the regions $U_{1}, U_{2}, U_{3}, U_{4}, U_{5}$ overlap in such a way that we can extract a strong Lyapunov covering $(\phi_{1}, V_{1})$, $(\phi_{2}, V_{2})$, \ldots, $(\phi_{5}, V_{5})$ subordinate to the Lyapunov covering $(\phi_{1}, U_{1})$, $(\phi_{1}, U_{2})$, \ldots, $(\phi_{5}, U_{5})$.  We now proceed onto Stage 2.  


\subsection{Stage 2}
We now prove that $\phi_{1}, \, \phi_{2}, \, \, \ldots, \, \phi_{5}$ constructed above can be ``patched" together to yield a globally-defined Lyapunov function $\Phi$ as in Lemma \ref{lem1}.  By Proposition \ref{vary_kappa_lemma}, to prove Lemma \ref{lem1} it is enough to show that for all $\alpha_{2}>\alpha_{1}$ there exists $\kappa_{2}=\kappa_{2}(\boldsymbol{\alpha})>0$ such that 
\begin{eqnarray*}
\mathcal{N}(\boldsymbol{a}, \boldsymbol{\alpha}, \boldsymbol{\kappa})\neq \emptyset
\end{eqnarray*}
for all $\boldsymbol{a}\in \mathbb{R}^{2}$ and all $\kappa_{1} \geq 0$.  In what follows, for all $\alpha_{2}>\alpha_{1}$, we construct $\Phi=\Phi_{\boldsymbol{a}, \boldsymbol{\alpha}}^{\boldsymbol{\kappa}}$ such that for some $\kappa_{2}=\kappa_{2}(\boldsymbol{\alpha})>0$:
\begin{eqnarray*}
\Phi_{\boldsymbol{a}, \boldsymbol{\alpha}}^{\boldsymbol{\kappa}}\in \mathcal{N}(\boldsymbol{a}, \boldsymbol{\alpha}, \boldsymbol{\kappa}), 
\end{eqnarray*}
for all $\boldsymbol{a}\in \mathbb{R}^{2}$ and for all $\kappa_{1} \geq 0$.  Hence the choice of $\kappa_{2}$ can depend on $\boldsymbol{\alpha}$ but not on $\boldsymbol{a}$ nor on $\kappa_{1}$.

To obtain $\Phi$ as above, we define smooth auxiliary functions  
\begin{eqnarray*}
\rho_{i,i+1}:U_{i}\cup U_{i+1}\rightarrow [0,1]
\end{eqnarray*}
such that $\rho_{i, i+1} = 0$ on $U_{i}$ and $\rho_{i, i+1}=1$ on $U_{i+1}$ for $i=1,\, 2, \ldots,\, 5$, and prove that 
\begin{eqnarray}\label{phipatch}
\phi_{i, i+1} :=\rho_{i, i+1} \phi_{i+1} + (1-\rho_{i, i+1}) \phi_{i}
\end{eqnarray}   
is a Lyapunov function on $U_{i} \cup U_{i+1}$ for $i =1, \, 2, \ldots, \, 5$.  Note then we can choose  $\Phi \in C^{\infty}(\mathbb{R}^{2})$ such that outside of a sufficiently large ball $B_{R}$, we have
\begin{eqnarray*}
\Phi= \begin{cases}
\phi_{i} & \text{ on } U_{i} \setminus (\cup_{j=1}^{4} (U_{j} \cap U_{j+1})), \, i =1, 2, \ldots, 5\\
\phi_{i, i+1}& \text{ on } U_{i} \cap U_{i+1}, \, i =1, 2, \ldots, 5.  
\end{cases}
\end{eqnarray*} 
It follows then by construction, 
\begin{enumerate}
\item $\Phi(z) \rightarrow \infty$ as $|z|\rightarrow \infty$.
\item  There exist constants $C, D>0$ such that $L\Phi(z) \leq -C \Phi(z) + D$ for all $z\in \mathbb{R}^{2}$.  
\end{enumerate} 

Note that by expression \eqref{phipatch}, for the newly defined $\phi_{i, i+1}$ we do not need to verify (I) and (II) in Definition \ref{lyapunovfunction}.  All we need to do is prove (III) on $U_{i}\cap U_{i+1}$.  To simplify matters later, we note that for $i=1,2,\ldots, 5$: 
\begin{eqnarray}\label{Lpatch}
 \, \, \, \, \, \, \, \, \, \, L \phi_{i, i+1}&=& \rho_{i, i+1} L \phi_{i+1} + (1- \rho_{i, i+1}) L\phi_{i} \\
\nonumber&\, & \, + \, (a_{1}x-\alpha_{1}x^{2}+y^{2})(\phi_{i+1}-\phi_{i})\partial_{x}\rho_{i, i+1}\\
\nonumber&\,& \, +\, (a_{2} y - \alpha_{2}xy)  ( \phi_{i+1}-\phi_{i})\partial_{y}\rho_{i, i+1}\\
\nonumber&\,& \, + \, \kappa_{1} (\phi_{i+1}-\phi_{i}) \partial_{xx} \rho_{i, i+1} + 2\kappa_{1} \partial_{x}( \phi_{i+1}-\phi_{2}) \partial_{x} \rho_{i, i+1}\\
\nonumber&\,& \, + \, \kappa_{2} (\phi_{i+1}-\phi_{i}) \partial_{yy} \rho_{i, i+1} + 2\kappa_{2} \partial_{y}( \phi_{i+1}-\phi_{2}) \partial_{y} \rho_{i, i+1},
\end{eqnarray}
on $U_{i} \cap U_{i+1}$.  


\begin{patch}
First we will patch together \(\phi_1\) and \(\phi_2\) to yield $\phi_{1,2}$, a Lyapunov function on $U_{1}\cup U_{2}$.  To do this, we first define the function $\rho_{1,2}$. Let $q:\mathbb{R}^{2}\rightarrow \mathbb{R}$ and $f, \, g:\mathbb{R}\rightarrow \mathbb{R}$ be defined by 
\begin{eqnarray*}
q(x,y)&=&|x|^{1/2}|y|-1,\\
   f(t)& =& \begin{cases}    
 \exp(\frac{-1}{1-(2t-1)^2}) & \text{if } t \in [0,1]\\      
 0 &\text{if } t \notin [0,1]   
\end{cases},\\
g(t)&=& \frac{1}{\|f\|_{1}}\int_{-\infty}^{t} f(s) \, ds.
\end{eqnarray*}
Let $\rho_{1,2}: U_{1}\cup U_{2}\rightarrow [0,1]$ be defined by
\begin{equation*}
\rho_{1,2}(x,y)=g(q(x,y)).  
\end{equation*} 
It is easy to see that $\rho_{1,2}\in C^{\infty}(U_{1} \cup U_{2})$ satisfies $\rho_{1,2}=0$ on $U_{1}$ and $\rho_{1,2}=1$ on $U_{2}$.  

\begin{claim}
Let \(\phi_{1,2}=\rho_{1,2}\phi_2+(1-\rho_{1,2})\phi_1\).  There exists \(\kappa_2(\boldsymbol{\alpha})>0\) such that $\phi_{1,2}$ is a Lyapunov function on $U_{1} \cup U_{2}$ 
 for all $\boldsymbol{a}\in \mathbb{R}^{2}$ and all $\kappa_{1}\geq 0$.  
\end{claim}
\begin{proof}
By estimates \eqref{estimate1} and \eqref{estimate2} and our choice of parameters, we have  
\begin{align}\label{dominant_term12}
\rho_{1,2}L\phi_2+(1-\rho_{1,2})L\phi_1\leq-\kappa_2E_1|x|^{(2+\sigma)\delta+1} + \max(D_{1}, D_{2}),
\end{align}
for some constant \(E_1>0\) independent of both $\boldsymbol{a}$ and $\boldsymbol{\kappa}$.  
For $i=1$ the remaining terms on the right-hand side of relation \eqref{Lpatch} are equal to  
\begin{eqnarray}\label{bound12}
&\,&-\alpha_1x^2(\phi_2-\phi_1)\partial_{x}\rho_{1, 2}-\alpha_2xy(\phi_2-\phi_1)\partial_{y} \rho_{1,2}\\
&\,&+\kappa_2(\phi_2-\phi_1)\partial_{yy} \rho_{1,2}+2\kappa_2\partial_y(\phi_2-\phi_1)\partial_{y} \rho_{1,2} + \mathcal{O}(|x|^{(2+ \sigma) \delta + 1/2})\notag
\end{eqnarray}
Since \(C_1>C_2\), we see that for sufficiently large $(x,y)\in U_{1}\cap U_{2}$:
 \begin{eqnarray*}
-\alpha_1x^2(\phi_2-\phi_1)\partial_{x}\rho_{1, 2}-\alpha_2xy(\phi_2-\phi_1)\partial_{y} \rho_{1,2}&\leq&-E_2f(q(x,y))|x|^{(2+\sigma)\delta+1}\\
(\phi_2-\phi_1)\partial^2_{y}\rho_{12}+2\partial_y(\phi_2-\phi_1)\partial_y\rho_{12}&\leq&\frac{E_3f(q(x,y))}{h(q(x,y))}|x|^{(2+\sigma)\delta+1}
\end{eqnarray*}
where \(h(t)=t^2(t-1)^2\) and \(E_2,E_3>0\) are independent of both $\boldsymbol{a}$ and $\boldsymbol{\kappa}$.    
Note that the quotient \(f(t)h(t)^{-1}\rightarrow 0\) as \(t\rightarrow 0\) or \(1\).  Therefore there exists \(\epsilon>0\) independent of $\boldsymbol{a}$ and $\boldsymbol{\kappa}$ such that for $|x|^{1/2}|y| \in (1, 1+\epsilon) \cup (2-\epsilon, \epsilon)$ we have
\begin{equation*}\label{subregion12}
\frac{E_3f(q(x,y))}{h(q(x,y))}<E_1/3.\notag
\end{equation*}
This implies that for sufficiently large \((x,y)\in U_{1} \cap U_{2}\) with $|x|^{1/2}|y| \in (1, 1+\epsilon) \cup (2-\epsilon, \epsilon)$, \eqref{bound12} is bounded by
\begin{align*}
-\kappa_2E_1/2|x|^{(2+\sigma)\delta+1}.
\end{align*}
Note that \(1/h(q(x,y))\) is bounded and  $f(q(x,y))$ bounded away from zero on $|x|^{1/2}|y| \in [ 1+ \epsilon, 2-\epsilon]$.  
From this, we see that there exists $\kappa_{2} > 0$ sufficiently small (which only depends on $\boldsymbol{\alpha}$) for the estimate
\[ E_2/2\geq \frac{E_3\kappa_2}{h(q(x,y))}\]
to hold on $|x|^{1/2} |y| \in [1+ \epsilon, 2-\epsilon]$. These estimates imply that for sufficiently large \((x,y)\in U_{1} \cap U_{2}\) with $|x|^{1/2} |y| \in [1+ \epsilon, 2-\epsilon]$, \eqref{bound12} is bounded by  
\begin{align*}
-C^\prime |x|^{(2+\sigma)\delta+1}
 \end{align*}
for some $C^\prime>0$. This finishes the proof of Claim 1.    
\end{proof}
\end{patch}


For the remainder of Stage 2, let \(k:\mathbb{R}\rightarrow\mathbb{R}\) be a smooth function such that \(k(t)=0\) for \(t\leq 0\), \(k(t)=1\) for \(t\geq 1\), and \(k'(t)> 0\) on \((0,1)\).

\begin{patch}
We will now patch \(\phi_2\) and \(\phi_3\) together to obtain a Lyapunov function on $U_{2} \cup U_{3}$.  To this end, let \(\rho_{2,3}(x,y)=k(|y|-1)\).  It is clear that $k$ can be chosen so that  $\rho_{2, 3} \in C^{\infty}(U_{1} \cup U_{2})$.  
\begin{claim}\label{patching_23}
Let $\phi_{2,3}=\rho_{2,3}\phi_3+(1-\rho_{2,3})\phi_2$.  Then for all $\boldsymbol{a}\in \mathbb{R}^{2}$ and all $\kappa_{1}\geq 0$, $\phi_{2,3}$ is a Lyapunov function on $U_{2} \cup U_{3}$.
\end{claim}
\begin{proof}
Since \(\rho_{2,3}\) is independent of \(x\), 
\begin{eqnarray*}\label{eq_23}
L \phi_{2,3}&=&\rho_{2,3}L\phi_3+(1-\rho_{2,3})L\phi_2+(a_2y-\alpha_2xy)(\phi_3-\phi_2)\partial_{y} \rho_{2,3}\\
&\,&\, +\, \kappa_2(\phi_3-\phi_2)\partial_{yy}\rho_{2,3}+2\kappa_2\partial_y(\phi_3-\phi_2)\partial_{y}\rho_{2,3}.\notag
\end{eqnarray*}
We chose \(C_3<C_2/D^\delta\) so that, for sufficiently large \(|x|\) with $(x,y) \in U_{2}\cap U_{3}$, 
\begin{equation*}
-\alpha_{2}xy(\phi_3-\phi_2) \partial_{y} \rho_{2,3}<0.
\end{equation*}
Note moreover that, by \eqref{estimate2p} and \eqref{estimate2p1}, on $U_{1} \cap U_{2}$, we have
\begin{eqnarray*}
L \phi_{2,3} \leq -D_1|x|^{2\delta+1}- \alpha_{2} xy (\phi_{3}-\phi_{2}) \partial_{y} \rho_{2,3} + \mathcal{O}(|x|^{2\delta}),  
\end{eqnarray*} 
which implies the claim.  
\end{proof} 
\end{patch}


\begin{patch}
We now patch together $\phi_{3}$ and $\phi_{4}$ to obtain a Lyapunov function on $U_{3} \cup U_{4}$.  Define \(\rho_{3,4}:\mathbb{R}^2\rightarrow \mathbb{R}\) by \(\rho_{3,4}(x,y)=k(2x+2)\).  
\begin{claim}\label{patching_34}
Let \(\phi_{3,4}=\rho_{3,4}\phi_4+(1-\rho_{3,4})\phi_3\). Then for all $\boldsymbol{a}\in \mathbb{R}^{2}$ and all $\kappa_{1}\geq 0$, \(\phi_{3,4}\) is a Lyapunov function on $U_{3}\cup U_{4}$.
\end{claim}
\begin{proof}
Since \(\rho_{3,4}\) is independent of \(y\), 
\begin{eqnarray*}\label{eq_34}
L \phi_{3,4}&=&\rho_{3,4}L\phi_4+(1-\rho_{3,4})L\phi_3+(a_1x-\alpha_1x^2+y^2)(\phi_4-\phi_3)\partial_{x} \rho_{3,4}\\
&\,&\, +\, \kappa_1(\phi_4-\phi_3)\partial_{xx} \rho_{3,4}+2\kappa_1\partial_x(\phi_4-\phi_3)\partial_{x}\rho_{3,4}.\notag
\end{eqnarray*}
By the choice of $N C_{4}<C_{3}$,  \(\phi_{4}-\phi_3\) is negative for sufficiently large \(y\) in the region.  Therefore, there exist $E_5,E_6,E_7>0$ such that
\begin{eqnarray}\label{est34}
L \phi_{3,4} &\leq& -\rho_{3,4}E_5y^2-((1-\rho_{3,4})E_6-E_7|\partial_{xx}\rho_{3,4}|)y^{2\delta(1-\sigma)} +\mathcal{O}(1).
\end{eqnarray}
Since $k,k^{\prime\prime}\rightarrow 0$ as $t\rightarrow 0$  there exists $\epsilon>0$ such that 
\begin{equation*}
(1-\rho_{3,4})E_6-E_7|\partial_{xx}\rho_{3,4}|>E_8
\end{equation*}
for all $(x,y)\in U_{3} \cap U_{4}$ with $x\in (-1, -1+\epsilon)$ and some $E_8>0$.  For $(x,y)\in U_{3} \cap U_{4}$ with $x\in[-1+\epsilon,-1/2)$, $\rho_{3,4}>E_9$ for some $E_9>0$ and $|\partial_{xx}\rho_{3,4}|$ is bounded.  Putting these two estimates together with \eqref{est34} finishes the proof of the claim.  
\end{proof}

\end{patch}


\begin{patch}

Lastly, we patch together \(\phi_4\) and \(\phi_5\) to obtain a Lyapunov function $\phi_{4,5}$ on $U_{4}\cup U_{5}$.  Define then \(\rho_{4,5}:\mathbb{R}^2\rightarrow\mathbb{R}\) by \(\rho_{4,5}(x,y)=k(\frac{2}{N}x-1)\). 

\begin{claim}\label{patching_45}
Let \(\phi_{4,5}=\rho_{4,5}\phi_5+(1-\rho_{4,5})\phi_4\).  Then for all $\boldsymbol{a}\in \mathbb{R}^{2}$ and all $\kappa_{1}\geq 0$, $\phi_{4,5}$ is a Lyapunov function on $U_{4} \cup U_{5}$.
\end{claim}
\begin{proof}
With the choice of $NC_{4}>C_{5}$ the proof is nearly identical to the proof of Claim 3, so we omit further discussion.    
\end{proof}
\end{patch}





\section{Regularity and Controllability}\label{positivity}

We now prove Lemma \ref{lem2}.  We first start by showing that the transition measures $P(t, z, \, \cdot\,)$ have smooth densities with respect to Lebesgue measure on $\mathbb{R}^{2}$.  This is a simple consequence of H\"{o}rmander's hypoellipticity theorem \cite{HOR}, a powerful result that says, roughly speaking, that if enough noise spreads throughout the system such regularity follows.  As we shall see, the word ``enough" can be quantified by considering the span of the Lie algebra generated by vector fields in the expression for the generator $L$.  In general, this requires rewriting the SDE in the Stratonovich form \cite{NOR}.  In the present case, It\^o and Stratonovich forms coincide, since the noise coefficients are constant.

\subsection{Regularity}

The regularity of the measures $P(t, z, \, \cdot \, )$ can be studied using either classical PDE theory \cite{HOR} or the Malliavin calculus \cite{NOR, ND}.  To state the result, for vector fields $X$ and $Y$, let $[X,Y]$ be their Lie bracket (commutator).  If we write the generator $L$ in the form:
\begin{eqnarray*}
L= X_{0} + X_{1}^{2}+ X_{2}^{2},
\end{eqnarray*}     
where $X_{0},\, X_{1},$ and $X_{2}$ are vector fields (treated as first-order differential operators, so that $X_j^2$ denotes a composition of such operator with itself) with smooth coefficients and if at each point $z\in \mathbb{R}^{2}$ the list:
\begin{align*}
&X_{j_{1}}(z) \, && j_{1}=1,2\\
&\left[X_{j_{1}}, X_{j_{2}} \right](z)\, &&j_{1}, j_{2}=0,1,2\\
& \left[ X_{j_{1}}, \left[X_{j_{2}}, X_{j_{3}}\right] \right](z) \, &&j_{1}, j_{2}, j_{3} =0,1,2\\
&\vdots \, && \vdots
\end{align*}
spans $\mathbb{R}^{2}$, then for $t>0$, $z\in \mathbb{R}^{2}$
\begin{eqnarray*}
P(t,z, dw)= p(t, z, w) \, dw,
\end{eqnarray*}
where $p(t, z, w)$ is a smooth function for $(t, z, w)\in (0, \infty) \times \mathbb{R}^{2} \times \mathbb{R}^{2}$.  Moreover, if $\nu$ is an invariant probability measure, then
\begin{equation*}
\nu( dw)= \rho_{\nu}(w) \, dw
\end{equation*}
where $\rho_{\nu}\in C^{\infty}(\mathbb{R}^{2}).$

\begin{proof}[Proof of Lemma \ref{lem2} \textbf{(1)}]
Note that 
\begin{eqnarray*}
L&=&\left((a_{1}x-\alpha_{1}x^{2}+y^{2})\frac{\partial}{\partial x}+ (a_{2}y -\alpha_{2} xy) \frac{\partial}{\partial y} \right) \\
&\,& \,+\, \left(\sqrt{\kappa_{1}} \frac{\partial}{\partial x} \right)^{2} + \left(\sqrt{\kappa_{2}} \frac{\partial}{\partial y} \right)^{2}.
\end{eqnarray*}
Thus let 
\begin{eqnarray*}
X_{0}&=&\left((a_{1}x-\alpha_{1}x^{2}+y^{2})\frac{\partial}{\partial x}+ (a_{2}y -\alpha_{2} xy) \frac{\partial}{\partial y} \right) \\
X_{1}&=&\sqrt{\kappa_{1}} \frac{\partial}{\partial x} \\
X_{2}&=&\sqrt{\kappa_{2}} \frac{\partial}{\partial y}. 
\end{eqnarray*}
If both $\kappa_{1}>0,\, \kappa_{2}>0$, then $X_{1}(z)$ and $X_{2}(z)$ span $\mathbb{R}^{2}$ for all $z\in \mathbb{R}^{2}$.  If $\kappa_{1}=0$, we see that
\begin{eqnarray*}
X_{2,2,0}:=\left[X_{2}, \left[ X_{2}, X_{0}\right] \right] =2 \kappa_{2} \frac{\partial}{\partial x}.
\end{eqnarray*}
Since $\kappa_{2}>0$, $X_{2,2,0}(z)$ and $X_{2}(z)$ span $\mathbb{R}^{2}$ for all $z\in \mathbb{R}^{2}$.  
\end{proof}


\subsection{Controllability}                                                              

We now prove Lemma \ref{lem2} \textbf{(2)} by utilizing a connection between SDEs and control theory provided by the Stroock-Varadhan support theorem \cite{SV}.  We will also use and modify results in \cite{AK, JK, JK1} to work in our setting.  

With fixed $a_{1}, a_{2} \in \mathbb{R}$, $\alpha_{2}>\alpha_{1}>0$, and $\kappa_{1}\geq 0, \, \kappa_{2}>0$, consider the ordinary differential equation:
\begin{eqnarray}\label{contsys}
\dot{x}_{t}&=& (a_{1}x_{t}-\alpha_{1} x_{t}^{2}+y_{t}^{2}) + \kappa_{1} u(t)\\
\dot{y}_{t}&=&(a_{2}y_{t}-\alpha_{2}x_{t}y_{t} ) + \kappa_{2} v(t), \nonumber  
\end{eqnarray}    
where $u,\, v: [0, \infty) \rightarrow \mathbb{R}$ are called \emph{controls}.  We assume $u$ and $v$ belong to $\mathcal{U}$, the class of piecewise constant mappings from $[0, \infty)$ into $\mathbb{R}$ with at most finitely many discontinuities.  For $w(t):=(u(t), v(t))$ with $u,\,v\in \mathcal{U}$, we let $z(z_{0}, w, t)$ be the maximal right integral curve of \eqref{contsys} passing through $z_{0}\in \mathbb{R}^{2}$ at $t=0$.

In view of \cite{AK}, one can use \emph{accessibility sets} of the system \eqref{contsys} to help determine the support of an extremal invariant probability measure.  More precisely, for $t\geq 0$, we let $A(z_{0}, \leq t)$ denote the \emph{accessibility set from $z_{0}$ in }$t$ \emph{units of time or less}, i.e., 
\begin{eqnarray*}
A(z_{0}, \leq t):=\{ z(z_{0}, w, s)\, :\, \, w=(u, v), \, u,v\in \mathcal{U}, \, 0\leq s \leq t\}.
\end{eqnarray*} 
Defining
\begin{eqnarray*}
A(z_{0}) := \bigcup_{t>0} A(z_{0}, \leq t),
\end{eqnarray*}    
we note the following result in \cite{AK}:
\begin{theorem}
If $\mu$ is an extremal invariant probability measure for the process $Z_{t}$, there exists $z_{0} \in \mathbb{R}^{2}$ such that 
\begin{eqnarray*}
\text{supp } \mu = \overline{A(z_{0})}.
\end{eqnarray*}
\end{theorem}

\begin{proof}
This follows by Proposition 1.1 and Definition 1.1 in \cite{AK} since $Z_t$ has continuous density with respect to Lebesgue measure.     
\end{proof}
Using Proposition 2.1 of \cite{AK} and Lemma \ref{lem2} \textbf{(1)}, to prove uniqueness of the invariant probability measure $\nu$ it is enough to show:
\begin{lemma}\label{quadcont}
For all $z_{0}, w_{0}\in \mathbb{R}^{2}$, there exists a non-empty open subset $U=U(z_{0}, w_{0})\subset \mathbb{R}^{2}$ such that:
\begin{equation}
\overline{A(z_{0})}\cap \overline{A(w_{0}})\supset U(z_{0}, w_{0}).
\end{equation}
\end{lemma}    

To prove the lemma above, we require ideas in \cite{JK, JK1}.  This is because it is difficult to find specific controls $u, \,v$ and then solve equation \eqref{contsys} to determine $\overline{A(z_{0})}$.  As emphasized in these works, one can simplify such a procedure for systems like \eqref{contsys} by using geometric control theory.  

To illustrate this approach, let $Z$ and $W$ be the vector fields on $\mathbb{R}^{2}$ determined by 
\begin{eqnarray*}
Z(x,y)&=& (a_{1}x-\alpha_{1} x^{2}+y^{2}, a_{2}y-\alpha_{2}xy)\\
W_{1}(x,y)&=& (\kappa_{1} ,0)\\
W_{2}(x,y)&=& (0, \kappa_{2}), 
\end{eqnarray*}               
and define a family of smooth vector fields:
\begin{eqnarray*}
F = \left\{ Z + u W_{1}+ vW_{2}\, : \, u, \, v\in \mathbb{R}\right\}.  
\end{eqnarray*}
We call $F$ a \emph{polysystem}.  If $\tilde{Z}\in F$, for $t\geq 0$ and $z_0\in \mathbb{R}^2$ let $\exp(t \tilde{Z})(z_0)$ denote its maximal integral curve passing through $z_0$ at $t=0$.  For $z_{0}\in \mathbb{R}^{2}$ and $t\geq 0$ we let $A_{F}(z_{0}, \leq t)$ denote the set of $w\in \mathbb{R}^{2}$ such that there exist $Z_{1}, Z_{2}, \ldots, Z_{k}\in F$ and times $t_{1}, t_{2}, \ldots, t_{k}\geq 0$ such that $t_{1}+t_{2}+ \cdots + t_{k}\leq t$ and 
\begin{equation*}
\exp(t_{k}Z_{k}) \circ \exp(t_{k-1} Z_{k-1}) \circ \cdots \circ \exp(t_{1} Z_{1}) (z_{0}) = w.  
\end{equation*}
It is easy to see that $A_{F}(z_{0}, \leq t) = A(z_{0}, \leq t)$ for all $z_{0}\in \mathbb{R}^{2}$, $t\geq  0$.

The benefit of using geometric ideas is that it provides a means by which to enlarge $F$ without changing $\overline{A_{F}(z_{0}, \leq t)}$.  In light of this, we say two polysystems $F_{1}$ and $F_{2}$ are \emph{equivalent}, denoted by $F_{1} \sim F_{2}$, if for all $z_{0}\in \mathbb{R}^{2}$ and $t>0$
\begin{equation*}
\overline{A_{F_{1}}(z_{0},\leq  t)}= \overline{A_{F_{2}}(z_{0}, \leq t)}.
\end{equation*}   
One can show, see \cite{JK, JK1}, that if $F_{1} \sim F$ and $F_{2} \sim F$, then $F_{1}\cup F_{2} \sim F$.  Thus to make $F$ as large as possible, we consider the union of all polysystems equivalent to $F$, called the \emph{saturate} of $F$, which we denote by $\text{Sat}(F)$.  To yield new equivalent polysystems from old, we require the following definitions and lemmata \cite{JK, JK1}.

\begin{definition}
We call a diffeomorphism $\eta: \mathbb{R}^{2} \rightarrow \mathbb{R}^{2}$ a \emph{normalizer} of the polysystem $F$ if for all $z_{0}\in \mathbb{R}^{2}$ and all $t>0$:
\begin{equation*}
\eta( \overline{A_{F}( \eta^{-1}(z_{0}), \leq t)}) \subset \overline{A_{F}(z_{0}, \leq t)}.  
\end{equation*}
We denote the set of all normalizers of $F$ by $\text{Norm}(F)$.  
\end{definition}
\begin{lemma}\label{normlem}
If $F$ is polysystem, then 
\begin{equation*}
\bigcup_{\eta \in \text{Norm}(F)} \left\{\eta_{*}(V) \, : \, V\in F \right\}\sim F,
\end{equation*}
where $\eta_{*}$ denotes the differential.  
\end{lemma}

\begin{lemma}\label{conelem}
If $F$ is a smooth polysystem, then $F$ is equivalent to the closed convex hull of $\{\lambda V\, : \, 0\leq \lambda \leq 1,\, V\in F\}$.  Here the closure is taken in the topology of uniform convergence with all derivatives on compact subsets of $\mathbb{R}^{2}$. 
\end{lemma}

With these ideas in place, we now prove Lemma \ref{lem2} \textbf{(2)}.  
\begin{proof}[Proof of Lemma \ref{lem2} \textbf{(2)}]
Set $F= \{ Z+ uW_{1}+ vW_{2} \, : \, u,v \in \mathcal{U}\}$.  By Lemma \ref{conelem}, for all constants $\lambda \in \mathbb{R}$:
\begin{eqnarray*}
 \lambda W_{1}&= &\lim_{n\rightarrow \infty} \frac{1}{n}\left(Z+ n \lambda W_{1} \right)\in \text{Sat}(F)\\
\lambda W_{2}&=& \lim_{n\rightarrow \infty} \frac{1}{n}\left(Z+ n \lambda W_{2} \right)\in \text{Sat}(F).
\end{eqnarray*}
It is easy to see that $z\mapsto \exp(\lambda W_{i})(z) \in \text{Norm}(F)$ for $\lambda \in \mathbb{R}$, $i=1,2$.  Since $\kappa_{2}>0$ Lemma \ref{normlem} implies:
\begin{equation*}
\exp(\lambda W_{2})_{*}(Z)= Z + \lambda \left[W_{2}, Z \right] + \frac{\lambda^{2}}{2}\left[ W_{2}\left[W_{2}, Z \right]\right]\in \text{Sat}(F).  
\end{equation*}  
for all $\lambda \in \mathbb{R}$.  Applying Lemma \ref{conelem} with $\mu \in \mathbb{R}$, we see that:
\begin{equation*}
\lim_{\lambda \rightarrow \infty}\frac{1}{\lambda^{2}}(\exp(\lambda |2\mu|^{1/2} W_{2})_{*}(Z))= |\mu| \left[ W_{2}\left[W_{2}, Z \right]\right] \in \text{Sat}(F).  
\end{equation*}
Moreover, $|\mu| \left[ W_{2}\left[W_{2}, Z \right]\right](z) =( |\mu| \kappa_{2}^{2},0).$  Using this vector field with $W_{2}$, we see that for $z_{0} \in \mathbb{R}^{2}$ and $t>0$:  
\begin{equation*}
\overline{A_{F}(z_{0}, \leq t)}\supset H(z_{0}),
\end{equation*}
where for $z=(x,y)$, $H(z) = \{w=(u,v)\in \mathbb{R}^{2}\, : \, u \geq x\}.$  This finishes the proof of Lemma \ref{quadcont} \textbf{(2)} since for all $z_0,\, w_0\in \mathbb{R}^2$
\begin{equation*}
\overline{A(z_0)}\cap \overline{A(w_0)}\supset U(z_0, w_0) \neq \emptyset,
\end{equation*}
for some $U(z_0, w_0)$ open.  

\end{proof}

 
 

\section{Instabilities}\label{explosion}

We assume now that $\alpha_{1}>\alpha_{2}$ and prove Theorem \ref{theorem2}.  Similar to Section \ref{lyapunov}, this will be established by constructing an appropriate test function $\Psi$, as in the hypotheses of Lemma \ref{psiexp}.  Before we proceed on to the construction of $\Psi$, we first discuss the deterministic dynamics under the assumption $\alpha_{1}>\alpha_{2}$.     

\subsection{A Robust Explosive Region}

It is not hard to see that when $\alpha_{2}>\alpha_{1}$ as in Section \ref{lyapunov}, the deterministic dynamics
\begin{eqnarray}\label{detdyn}
\dot{x}(t)&=& a_{1}x(t)-\alpha_{1} x(t)^{2}+y(t)^{2}\\
\dot{y}(t)&=& a_{2}y(t)-\alpha_{2}x(t)y(t)\nonumber 
\end{eqnarray}
has a single unstable trajectory along the negative $x$-axis.  When $\alpha_{1}>\alpha_{2}$, we prove this instability in equation \eqref{detdyn} is more robust.  Remaining consistent with the earlier notation, we let $z(t)=(x(t), y(t))$ be the solution to equation \eqref{detdyn}, defined on the maximal interval of times $t>0$ on which it exists.  The initial condition will always be made clear in the following arguments.  We first need a definition.
\begin{definition}
We call a set $U\subset \mathbb{R}^{2}$ \emph{invariant} if for all $z(0)=(x_{0},y_{0})\in U$, $z(t)\in U$ for all times $t\geq 0$ for which the solution $z(t)$ is defined.    
\end{definition}          
Using Lyapunov-type criteria, we will determine an invariant set $A^{0}\subset \mathbb{R}^{2}$ for which solutions starting at $z_{0}\in A^{0}$ escape to infinity in finite time.  To this end, we recall that the infinitesimal operator for \eqref{detdyn} has the form:
\begin{equation*}
L^{0}=( a_{1}x-\alpha_{1} x^{2}+y^{2}) \frac{\partial}{\partial x}+ ( a_{2}y -\alpha_{2} xy) \frac{\partial}{\partial y}.  
\end{equation*}
For $\xi, M>0$, define functions $\psi_{1}, \psi_{2}, \psi_{3}: \mathbb{R}^{2}\rightarrow \mathbb{R}$ by:
\begin{eqnarray*}
\psi_{1}(x,y)&=& -x\\
\psi_{2}(x,y)&=& \xi y-x\\
\psi_{3}(x,y)&=& -\xi y-x, 
\end{eqnarray*}
and let 
\begin{equation*}
U_{\xi, M}= \psi_{1}^{-1}((M, \infty)) \cap \psi_{2}^{-1}((0, \infty))\cap \psi_{3}^{-1}((0, \infty)). 
\end{equation*}

\begin{proposition}
For all $\xi>0$ such that $\xi^{-2}< \alpha_{1}-\alpha_{2}$, there exists $M>0$ such that $U_{\xi, M}$ is invariant.     
\end{proposition}
\begin{proof}
Since $0<\xi^{-2}< \alpha_{1}-\alpha_{2}$, we may choose $M>0$ such that for all $j=1,2,3$:
\begin{equation*}
L^{0} \psi_{j}(x,y) > 0,
\end{equation*} 
for all $(x,y) \in U_{\xi, M}$.  Fix $z(0)=(x_{0}, y_{0})\in U_{\xi, M}$ and let
\begin{equation*}
S=\inf_{t>0}\{z(t) \in \mathbb{R}^{2}\setminus U_{\xi, M} \}.
\end{equation*}
We have for $t_{1} < t_{2}$, $t_{1}, t_{2} \in [0, S)$:
\begin{equation*}
-x(t_{2})+x(t_{1}) =\psi_{1}(z(t_{2}))-\psi_{1}(z(t_{1}))=\int_{t_{1}}^{t_{2}} L^{0} \psi_{1}(z(u)) \, du >0.
\end{equation*}
Thus $x(t)$ is strictly decreasing on $[0, S)$.  By continuity, $z(t)$ cannot exit $U_{\xi, M}$ through the vertical line $x=M$.  Using similar reasoning, we see that $\psi_{2}(z(t))$ and $\psi_{3}(z(t))$ are strictly increasing on $[0, S)$.  This implies that $z(t)$ cannot exit $U_{\xi, M}$ through the lines $|x|=\xi |y|$.           
\end{proof}

\begin{figure}
\centering
\caption{The region $U_{1, 15}$ and approximate dynamics when $\alpha_{1}>\alpha_{2}$.}\includegraphics[width=4in]{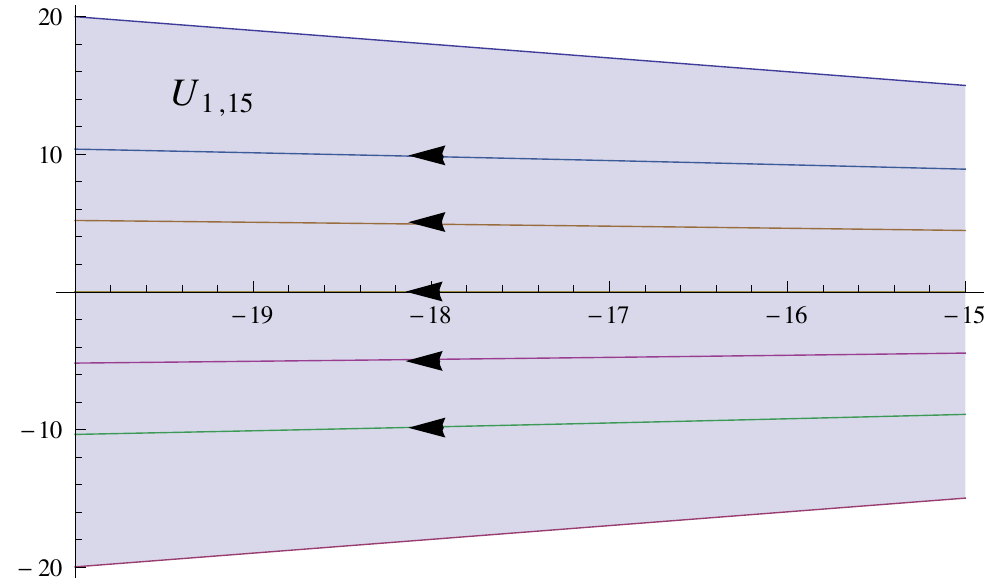}
\end{figure}

For $\xi>0$ such that $\xi^{-2}<\alpha_{1}-\alpha_{2}$, choose $M>0$ so that the conclusion of the preceding proposition is satisfied and 
\begin{equation}\label{ldetbound}
L^{0} \psi_{1}(x,y) \geq C [\psi_{1}(x,y)]^{2},
\end{equation}
for all $(x,y) \in U_{\xi, M}$ for some $C>0$.  Let $A^{0}=U_{\xi, M}$.  We now have the following proposition, illustrated by Figure 5.6:   
\begin{proposition}
For all $z(0)=(x_{0}, y_{0}) \in A^{0}$, $z(t)$ reaches infinity in finite time $T>0$ where 
\begin{equation*}
T \leq \frac{1}{C|x_{0}|}.
\end{equation*}  
\end{proposition}

\begin{proof}
Fix $z(0)=(x_{0}, y_{0})\in A^{0}$ and let $T_{n}=\inf_{t>0} \{z(t) \in B_{n}^{c}\}$ where $B_{n}$ is the open ball of radius $n$ centered at the origin.  Let $T$ be the finite or infinite limit of $T_{n}$ as $n\rightarrow \infty$.  By the previous proposition, $z(t) \in A^{0}=U_{\xi, M}$ for all $t\in [0, T)$.  Fix $t\in [0, T)$.  By the proof of the previous proposition, the map $s\mapsto x(s)$ is strictly decreasing and smooth.  As such, it is a diffeomorphism onto its image.  Using $x^{-1}$ to denote its smooth inverse, we have:
\begin{eqnarray}\label{ineqdet}
t&=& x^{-1}( x(t)) - x^{-1}(x(0))\nonumber\\
&=& \int_{x(0)}^{x(t)}\frac{d}{d y} x^{-1}(y) \, dy\nonumber \\
&=& \int_{x(0)}^{x(t)}\frac{1}{ \dot{x}(x^{-1}(y))}\, dy \nonumber\\
&=&\int_{x(t)}^{x(0)} \frac{1}{-\dot{x}(x^{-1}(y))}\, dy \nonumber\\
&< & \int_{x(t)}^{x(0)} \frac{1}{C (x(x^{-1}(y))^{2}}\, dy  \\
&=& \frac{1}{C}\left[ \frac{1}{|x_{0}|}-\frac{1}{|x(t)|}\right]\nonumber\\
&\leq &\frac{1}{C|x_{0}|}\nonumber,
\end{eqnarray}
where \eqref{ineqdet} follows from \eqref{ldetbound}.  Since $t\in [0, T)$ was arbitrary, we see that 
\begin{equation*}
T \leq \frac{1}{C|x_{0}|}.  
\end{equation*}

\end{proof}

\subsection{The Random Dynamics}  The robustness of the explosive region $U_{\xi, M}$ suggests that noise will not be sufficient to stabilize the system, a fact which we will now prove.  Let $q:(-\infty,0)\times\mathbb{R}\rightarrow\mathbb{R}$, \(f:\mathbb{R}\rightarrow \mathbb{R}\), \(\rho:(-\infty,0)\times\mathbb{R}\rightarrow\mathbb{R}\), and \(h:\mathbb{R}^2\rightarrow \mathbb{R}\) by
\begin{displaymath}  
q(x,y)=\frac{\xi y}{x},\hspace{3mm} f(t) = \left\{     
\begin{array}{lr}      
 \exp(\frac{-1}{1-t^2}) & \text{if } t \in [-1,1]\\      
 0 &\text{if } t \notin [-1,1]   
  \end{array}   
\right. ,\end{displaymath}  
\begin{displaymath}  
\rho=f\circ q, \text{ and } h(x,y) = \left\{     
\begin{array}{lr}      
 \exp(\frac{1}{x+M}) & \text{if } x<-M\\      
 0 &\text{if } x>-M   
  \end{array}   
\right. .\end{displaymath}  
Finally, let \(g=h\rho\) (the product of $h$ and $\rho$) and note that \(g\) is smooth on \(\mathbb{R}^2\), non-negative, bounded, and is supported in \(\overline{U_{\xi, M}}\).  Moreover, $g$ is strictly positive on $U_{\xi, M}$.  We will use \(g\) to show that the system in not regularized by noise in the specified parameter range.

\begin{proposition}
For all \(\xi>0\) such that \((\alpha_1-\alpha_2)/2<\xi^{-2}<\alpha_1-\alpha_2\) there exist constants \(M,C>0\) such that \(Lg\geq Cg\) on \(\mathbb{R}^2\) and $g$ is strictly positive on $U_{\xi, M}$.  If \(a_1\geq a_2\) then \(M\) can be chosen independent of \(\xi\).
\end{proposition}

\begin{proof}
First note that the function $h$ only depends on $x$; hence, applying $L$ to $g$ we obtain:  
\begin{eqnarray*}
Lg&=&\rho Lh+(a_1x-\alpha_1x^2+y^2)h\partial_x\rho+(a_2y-\alpha_2xy)h\partial_y\rho+\kappa_1h\partial_{xx}\rho\\
&\,&\,+\, \kappa_2h\partial_{yy}\rho+2\kappa_1\partial_x\rho\partial_xh.
\end{eqnarray*}
To estimate $Lg$, we find it convenient to use the following expressions for the derivatives above:
\begin{align*}
&Lh=\left(\frac{\alpha_1x^2-a_1x-y^2}{(x+M)^2}+\kappa_1\frac{2(x+M)+1}{(x+M)^4}\right)h,\\
&\partial_xh=\frac{-1}{(x+M)^2}h, \hspace{5mm}\partial_x\rho=\frac{2q^2}{x(1-q^2)^2}\rho,\hspace{6mm} \partial_y\rho=\frac{-2q^2}{y(1-q^2)^2}\rho, \notag\\
&\partial_{xx}\rho=\frac{2q^2(q^4+4q^2-3)}{x^2(1-q^2)^4}\rho,\hspace{5mm} \partial_{yy}\rho=\frac{2\xi^2(3q^4-1)}{x^2(1-q^2)^4}\rho,
\end{align*}
all of which follow from the definitions of $h$ and $q$.  Let \(D_1>0\) and \(D_2>0\) be upper bounds for 
\[-\frac{2t^2(t^4+4t^2-3)}{(1-t^2)^4}\text{ and } -\frac{2(3t^4-1)}{(1-t^2)^4}\]
 respectively for \(t\in (-1,1)\).  Therefore, since \(\partial_x\rho\partial_xh\geq0\) and $\xi^{-2} >(\alpha_{1}-\alpha_{2})/2$, we have the following estimate for the second-order terms
 \begin{eqnarray*}
 \kappa_1h\partial_{xx}\rho+\kappa_2h\partial_{yy}\rho+ 2\kappa_{1} \partial_{x} \rho \partial_{x} h&\geq&-(\kappa_1D_1+\kappa_2D_2\xi^2)x^{-2}g \\
 &\geq&-\left(\kappa_1D_1+\frac{2\kappa_2D_2}{\alpha_1-\alpha_2}\right)x^{-2}g,
 \end{eqnarray*}
 on $U_{\xi, M}$.  
To handle the remaining terms we first recall that \(y^2x^{-2}\leq\beta^{-2}\) on \(U_{\xi,M}\).  Hence,
\begin{eqnarray*}
(a_{1}x-\alpha_1x^2+y^2)h\partial_x\rho+ (a_{2}y-\alpha_2xy)h\partial_y\rho &\geq& g\frac{2q^{2}|x|}{(1-q^{2})^{2}}\bigg(\alpha_1-\alpha_2\\
&\, & \,-\, \xi^{-2}+|x|^{-1}(a_{1}-a_{2})\bigg),
\end{eqnarray*}
and
\begin{equation*}
\rho Lh\geq \left(\frac{(\alpha_1-\xi^{-2})x^2-a_1x}{(x+M)^2}+\frac{2\kappa_1(x+M)+1}{(x+M)^4}\right)g.
\end{equation*}
Let \(D_3>0\) be an upper bound for 
\begin{equation*}
-\frac{2\kappa_1(x+M)+1}{(x+M)^2}
\end{equation*}
on \((-\infty,-M)\).  Therefore
\begin{equation*}
\rho Lh\geq\left(\alpha_2-|a_1|M^{-1}-D_3M^{-2}\right)\frac{x^2}{(x+M)^2}g.
\end{equation*}
Together these estimates imply:
\begin{align*}
Lg\geq&\left[\left(\alpha_2-|a_1|M^{-1}-D_3M^{-2}-\left(\kappa_1D_1+\frac{2\kappa_2D_2}{\alpha_1-\alpha_2}\right)M^{-2}(1+M/x)^2\right)\frac{x^2}{(x+M)^2} \right.\\
&\left. +\frac{2|x|q^2(\alpha_1-\alpha_2-\xi^{-2}+|x|^{-1}(a_1-a_2))}{(1-q^2)^2}\right]g.
\end{align*}
For all \(\eta>1\) and $M$ sufficiently large (the choice of which does not depend on $\xi$):
\begin{align*}
Lg\geq&\left(\frac{\alpha_2}{\eta (1+M/x)^2}+\frac{2|x|q^2(\alpha_1-\alpha_2-\xi^{-2}+(a_1-a_2)|x|^{-1})}{(1-q^2)^2}\right)g
\end{align*}
for all \(\xi\) satisfying \(\alpha_1-\alpha_2>\xi^{-2}>(\alpha_1-\alpha_2)/2\).  If $a_{1}\geq a_{2}$ the second term on the right is nonnegative; thus we have proved the proposition when $a_{1}\geq a_{2}$.  If $a_{2}>a_{1}$, we have
\begin{align*}
Lg\geq&\left(C_M+\frac{2y^2\xi^2(\alpha_1-\alpha_2-\xi^{-2}-|a_2-a_1|M^{-1})}{|x|(1-q^2)^{2}}\right)g,
\end{align*}
for some constant $C_{M}>0$.  For a given \(\xi\) satisfying \(\xi^{-2}<\alpha_1-\alpha_2\) the second term is positive for sufficiently large \(M\), thereby proving the proposition in this case as well.
\end{proof}

\section{Conclusions}

We saw that in equation \eqref{initial_system}, $\alpha_{1}=\alpha_{2}$ represents a critical barrier for ergodicity.  In particular, it was shown that if $\alpha_{2}>\alpha_{1}$, there exists a unique invariant probability measure.  This result was essentially proven by the existence of a globally-defined Lyapunov function $\Phi$ such that $L\Phi$ satisfies the strong bound \eqref{boundL}.  The function $\Phi$ was constructed by first exhibiting a strong Lyapunov covering $\{ (\phi_{i}, V_{i})\}_{i=1}^{5}$, the existence of which, at the very least, guarantees that $Z_{t}$ cannot exit $\mathbb{R}^{2}$ directly through any $V_{j}$.  To eliminate the (highly unlikely) possibility of other explosion scenarios, the covering was patched together to yield $\Phi$.          

In the case when $\alpha_{2}>\alpha_{1}$ and $\kappa_{1}>0$, we were able to quantify the rate of convergence to the steady state.  As we saw, convergence is exponentially fast in a norm stronger than the total variation norm on $\mathcal{B}$-measures.  If $\kappa_{1}=0$, ergodicity still remains true in this case, but it is much harder to quantify the rate of convergence.  This should be expected, since there is less noise in the system.  Methods developed in \cite{AK, HOR, JK, JK1,SV} proved useful in showing ergodicity in this situation.

In the case when $\alpha_{1}>\alpha_{2}$, we saw that ergodicity cannot even be discussed, since global stability is not satisfied.  This is because solutions starting in a wedge-like region containing the negative $X$-axis explode in finite time with positive probability.  We believe that this is primarily due to the presence of a more robust explosive region in the deterministic dynamics.       

\section{Acknowledgements}

We were introduced to the general circle of problems this paper addresses by K. Gaw\c edzki.  We thank M. Hairer for suggesting \cite{JK1} and J. Zabczyk for a reference to \cite{SM}.  J. W. was partially supported by the NSF grant DMS 1009508.  D. H. gratefully acknowledges support from NSF VIGRE grant through the
Mathematics Graduate Program at the University of Arizona.
\newpage

\section{Appendix}

In this section, we list the parameters introduced in Stage 1 of Section \ref{lyapunov} along with their chosen value.  Recall that $\alpha_{2}>\alpha_{1}>0$ are fixed constants.      
\begin{eqnarray*}
\sigma&=& \frac{\alpha_{1}+\alpha_{2}}{2\alpha_{2}}\in ( \alpha_{1}/\alpha_{2}, 1)\\
\delta& =& \frac{\kappa_{2}}{8\alpha_{1}}\\
\beta &=& (2+\sigma) \delta=\frac{(\alpha_{1}+3\alpha_{2})}{16 \alpha_{1} \alpha_{2}}\kappa_{2}\\
\gamma&=&(1-\sigma) \delta =  \frac{(\alpha_{2}-\alpha_{1})}{16 \alpha_{1} \alpha_{2}}\kappa_{2}\\
\eta&=&\frac{\alpha_2}{2}\\
D&=& 1 + \frac{(\alpha_{2} + 2|a_{2}|)(\alpha_{2}-\alpha_{1})}{\alpha_{2}}> \max(1, (\alpha_{2}+2|a_{2}|)(1-\sigma))\\
N&=& 1+ \frac{4|a_{2}|}{\alpha_{2}}> \max(1, 4|a_{2}|/\alpha_{2})\\
C_{1} &=& 2\\
C_{2}&=& 1\\ 
C_{3}&=& \frac{1}{2D^{\delta}}\\
C_{4}&=&\frac{1}{3ND^{\delta}}\\
C_{5}&=& \frac{1}{4D^{\delta}}.
\end{eqnarray*}
With these choices, choose $\kappa_{2}=\kappa_{2}(\boldsymbol{\alpha})>0$ such that:
\begin{eqnarray*}
&\,&\delta,\, \gamma \in(0,1/2),\\
&\,&(\alpha_{1}-\sigma \alpha_{2}) + \kappa_{2}\sigma( 2 \sigma \delta +1) <0,\\
&\,&\frac{E_{3} \kappa_{2}}{h(q(x,y))}\leq E_{2}/2 \text{ for all } |x|^{1/2}|y| \in [1+\epsilon, 2-\epsilon],
\end{eqnarray*}
where $\epsilon, E_{2}, E_{3}>0$ and $h(q(x,y))$ were introduced in Patch 1 of Stage 2 and are all independent of $\boldsymbol{\kappa}$ and $\boldsymbol{a}$.  It is easy to see that $\kappa_{2}$ only depends on $\boldsymbol{\alpha}$ which is permissible by Proposition \ref{vary_kappa_lemma}.

\end{document}